\documentclass[letterpaper, 11pt]{article}

\usepackage[bindingoffset=0.2in,left=1in,right=1in,top=1in,bottom=1in,footskip=.25in]{geometry}
\usepackage{authblk}

\usepackage{mathtools}
\usepackage{amssymb}
\usepackage{comment}
\usepackage{enumerate}
\usepackage{amsmath,amsthm}
\usepackage[T1]{fontenc}
\usepackage[utf8]{inputenc}
\usepackage[toc,page]{appendix}
\DeclareMathOperator{\Bin}{\mathrm{Bin}}

\DeclarePairedDelimiter{\abs}{\lvert}{\rvert}

\newtheorem{theorem}{Theorem}

\newtheorem{lemma}[theorem]{Lemma}
\newtheorem{remark}[theorem]{Remark}
\newtheorem{corollary}[theorem]{Corollary}

\newtheoremstyle{TheoremNum}
{\topsep}{\topsep}              
{\itshape}                      
{}                              
{\bfseries}                     
{.}                             
{ }                             
{\thmname{#1}\thmnote{ \bfseries #3}}
\theoremstyle{TheoremNum}

\usepackage{color}
\usepackage[usenames,dvipsnames]{xcolor}

\newcommand{\Exp}{\ensuremath{\textrm{Exp}}} 
\newcommand{\EE}{\ensuremath{\mathbb{E}}}     
\newcommand{\N}{\ensuremath{\mathbf{N}}}

\newcommand{\threshold}{\ensuremath{\mathfrak{a}_{\textsf{th}}}} 
\newcommand{\thebox}{\ensuremath{\Lambda}} 
\newcommand{\start}{\ensuremath{\mathfrak{a}}} 
\let\eps=\varepsilon

\title{Bootstrap percolation with inhibition}
\author[1]{Hafsteinn Einarsson\thanks{hafsteinn.einarsson@inf.ethz.ch; author
was supported by grant no.\ 200021 143337 of the Swiss National Science Foundation.}}
\author[1]{Johannes Lengler\thanks{johannes.lengler@inf.ethz.ch}}
\author[1]{Frank Mousset\thanks{frank.moussetf@inf.ethz.ch; author was
supported by grant no.\ 6910960 of the Fonds National de la Recherche,
Luxembourg.}}
\author[2]{Konstantinos Panagiotou\thanks{kpanagio@math.lmu.de}}
\author[1]{Angelika Steger\thanks{asteger@inf.ethz.ch}}
\affil[1]{Department of  Computer Science, ETH Zurich, Switzerland}
\affil[2]{University of Munich, Mathematics Institute, M\"{u}nchen, Germany}

\begin{document}
\maketitle

 

\begin{abstract}
  Bootstrap percolation is a prominent framework for studying the spreading of
  activity on a graph. We begin with an initial set of active vertices. The
  process then proceeds in rounds, and further vertices become active as soon as
  they have a certain number of active neighbors. A recurring feature in
  bootstrap percolation theory is an `all-or-nothing' phenomenon: either the size
  of the starting set is so small that the process stops very soon, or it
  percolates (almost) completely.

  Motivated by several important phenomena observed in various types of
  real-world networks we propose in this work a variant of bootstrap
  percolation that exhibits a vastly different behavior. Our graphs have two
  types of vertices: some of them obstruct the diffusion, while the others
  facilitate it. We study the effect of this setting by analyzing the process
  on Erd\H{o}s-R\'enyi random graphs. Our main findings are two-fold. First we
  show that the presence of vertices hindering the diffusion does not result in
  a stable behavior: tiny changes in the size of the starting set can
  dramatically influence the size of the final active set. In particular, the
  process is non-monotone: a larger starting set can result in a smaller final
  set. In the second part of the paper we show that this phenomenom arises from
  the round-based approach: if we move to a continuous time model in which
  every edge draws its transmission time randomly, then we gain stability, and
  the process stops with an active set that contains a non-trivial constant
  fraction of all vertices. Moreover, we show that in the continuous time model
  percolation occurs significantly faster compared to the classical round-based
  model. Our findings are in line with empirical observations and demonstrate
  the importance of introducing various types of vertex behaviors in the
  mathematical model.
\end{abstract}

\section{Introduction}

\emph{Bootstrap percolation} is a classical and well-studied mathematical
framework for the spreading of activity on a given graph. One starts with an
initial set of {\em active} vertices; this set may be chosen randomly or
deterministically. The process then proceeds in rounds, and further vertices
become active as soon as they have at least $k$ active neighbors, where $k \in
\mathbb{N}$ is a parameter of the process. The process is said to \emph{percolate}
if all vertices eventually turn active.

This process was first studied in 1979 on Bethe lattices by Chalupa, Leath and
Reich~\cite{chalupa1979bootstrap} to model
demagnetisation in magnetic crystals. If we choose the starting set randomly,
then one would expect that for many graphs there is a \emph{percolation threshold}
such that if the number of starting
vertices is above this threshold, then the process percolates, whereas
if it is below the threshold, it does not.
For example, such a threshold has been determined for finite grids by
Balogh, Bollob{\'a}s, Duminil-Copin and Morris~\cite{balogh2012sharp} and
for the Erd\H{o}s-R\'enyi random graph
by Janson, {\L}uczak, Turova, and Vallier~\cite{Janson2010}.
The problem has also been studied on various other
graphs~\cite{aizenman1988metastability,balogh2007bootstrap,balogh2006bootstrap}
and models, as for example in cellular automata~\cite{schonmann1992behavior,
Holroyd2007}. In all of these examples we observe an ``all-or-nothing''
phenomenon: if the size of the starting set is too small, then the process stops
rather quickly, and otherwise, it spreads to a level that includes (almost) all
vertices of the underlying graph.
In a way, this is not too surprising, as the likelihood that a vertex has $k$ active
neighbors increases with the total number of active vertices.

The aim of this paper is to study percolation processes with inhibition that
can restrict further dissemination of activity. As an example, consider the
diffusion of an innovation in a society. When a new product is introduced to
the market people may like and promote it or they may dislike and denigrate it.
If we now assume that people buy the product as soon as they get, say, $k$ more
positive feedbacks than negative ones from their neighbors, we have a bootstrap
percolation process with inhibition. Another example is a phenomenon called
\emph{input normalization} in neuroscience, cf.\ \cite{Carandini2012} for a
review. This refers to the following well-studied observation: when a signal
activates a small part of a local ensemble of neurons, the activity spreads
through to recurrent connections. But only up to a certain point. Then
inhibitory neurons are strong enough to stop a further spread of activitation.
In this way, very different input strengths can lead to similar levels of
activity that never surpass a certain upper bound. Such an effect has been
observed experimentally in many species~\cite{Heeger2009, Olsen2010,
Louie2011}.

In this paper, we consider the Erd\H{o}s-R\'enyi random graph $G_{n,p}$ with two types of
vertices: \textit{inhibitory} vertices (those obstructing the diffusion) and
\textit{excitatory} vertices (those facilitating the diffusion). First we show
that in the standard, round-based percolation model the introduction of
inhibitory vertices does not result in a stable behavior: either inhibition has
essentially no effect, or tiny changes in the size of the starting set can
dramatically influence the size of the final active set. In particular, the
process is non-monotone: a larger starting set can result in a smaller final
set. In the second part of the paper we show that such a phenomenon is actually
the result of the round-based approach: if we move to a continuous time model
in which every edge draws its transmission time randomly according to an
exponential distribution, then normalization is an automatic and intrinsic
property of the process. Moreover, we find that random edge delays accelerate
percolation dramatically: for transmission delays that are distributed
according to independent exponential distributions with mean one the time to
activate all vertices reduces from $\Theta(\log\log n)$ in the round based
model to $O(1)$ time in the asynchronous model.

\subsection{Model and results}   
The classical bootstrap percolation process on Erd\H{o}s-R\'enyi random graphs
was studied by Janson, {\L}uczak, Turova, and Vallier~\cite{Janson2010}.  This
process starts with a random active subset of size $\start$ of the vertices.
The process then proceeds in rounds, where in each round all non-active
vertices that have at least $k$ active neighbors also become active, and remain
so forever. A percolation process  \emph{percolates completely} if there is
some round in which every vertex is active. It \emph{almost percolates} if
there is some round in which $n-o(n)$ vertices are active. Janson et al.~showed
that for $1/n \ll p \ll n^{-1/k}$
there exists a threshold
\[
\threshold(n,p,k) = (1-1/k)\left(\frac{(k-1)!}{np^k}\right)^{1/(k-1)}
\]
such that for every $\eps >0$, a.a.s.\footnote{Asymptotically almost surely,
that is, with probability tending to one as $n \to \infty$.}~the process almost
percolates for $\start > (1+\eps)\threshold$, and a.a.s.~it stays forever at
$O(\threshold)$ active vertices if $\start < (1-\eps)\threshold$. They also
showed that for starting sets of size $\start = (1+\eps)\threshold$ the process
almost percolates in $\log_k \log (pn) +O(1)$ rounds, where the hidden constant
depends only on $\eps$. Observe that this result immediately carries over to
{\em directed} random graphs (in which activation requires $k$ active
in-neighbours) if we insert each directed edge with probability $p$.

We extend the standard model by allowing inhibitory vertices: we assume that
each of the $n$ vertices is  \emph{inhibitory} with probability $\tau$ and
\emph{excitatory} with probability $1-\tau$, independently. To be slightly more general,
we also introduce an additional parameter $0<\gamma \le 1/p$ and insert each directed edge with
excitatory (inhibitory) origin independently with probability $p$ (with
probability $\gamma p$). The process is similar to the classical bootstrap
percolation with one crucial difference: a previously inactive vertex $v$ turns
active in some round $i$ if after round $i-1$ the number of active excitatory
neighbors of $v$ \emph{exceeds} the number of its active inhibitory neighbors
by at least~$k$.
We generalize the threshold function $\threshold$ so that it now also depends on $\tau$:
\begin{equation}\label{eq:a0} \threshold = \threshold(n,p,k,\tau) = (1-1/k)\left(\frac{(k-1)!}{(1-\tau)^knp^k}\right)^{1/(k-1)}.\end{equation}
Note that the threshold does not depend on the inhibition excess $\gamma$. Note further that, compared to the threshold $\threshold(n,p,k,0)$ for the case without inhibition, there is an additional factor of $(1-\tau)^k$ in the denominator. This factor can be interpreted in the following way: clearly, a necessary condition for percolation is that the process percolates in the subgraph induced by the excitatory vertices, which has $(1+o(1))(1-\tau)n$ vertices a.a.s.. If we choose a random starting set of size $\start$, then this starting set will contain $(1+o(1))(1-\tau)\start$ excitatory vertices a.a.s.. Then, by the result for the process without inhibition, the process will not percolate if
\[ (1-\tau)\start \leq  (1-\eps)\threshold((1-\tau)n,p,k,0),\]
or, equivalently, if $\start \leq (1-\eps)\threshold(n,p,k,\tau)$. In particular, we can restrict our
analysis to the case $\start \geq \threshold$, since, by the results of Janson,
\L{}uczak, Turova and Vallier~\cite{Janson2010}, the process with $\start \leq
(1-\eps)\threshold$ will stop with $O(\start)$ active vertices.

Our results for this process are collected in the following theorem.

\begin{theorem}\label{thm:synchron}
  Let $\eps, \tau, \gamma >0$, $k\ge 2$ and assume $1/n \ll p \ll n^{-1/k}$ and
  $\start \geq (1+\eps)\threshold(n,p,k,\tau)$. Then the  bootstrap percolation
  process with inhibition a.a.s.\ satisfies the following.
  \begin{enumerate}[(i)]
    \item For $\tau < 1/(1+\gamma)$ the process almost percolates
      in $\log_k\log_{(\start/\threshold)}(np)+O(1)$ rounds (as it does in the
      case without inhibition).
      If, additionally, $p=\omega(\log n / n)$, then the process percolates completely
      in the same number of rounds.
    \item For $\tau > 1/(1+\gamma)$  and $\start \ge (\log n)^{2+\eps}$ and $p
      = \omega(\log n/n)$ the process is \emph{chaotic} in the following sense:
      for every constant $C_1 >0$ there exists a constant $C_2> C_1$ such that
      for every target function $f$ with $(\log n)/p \ll f(n) \ll n$, there
      exists a function $c\colon \mathbb{N}\to [C_1,C_2]$ such that if one starts the
      process with $\start = \lfloor c(n) \threshold\rfloor$ vertices, then it
      stops with $a^* =(1+o(1))f(n)$ active vertices a.a.s..
  \end{enumerate}
\end{theorem}
In other words, if $\tau < 1/(1+\gamma)$ 
then inhibition has basically no effect on
the outcome of the process: the process behaves similar as in the classical
case with $(1-\tau)n$ vertices and no inhibition. On the other hand,  if $\tau
> 1/(1+\gamma)$, then the outcome of the process depends in a rather unstable
way on the size of the initially active set: by changing the size of the
starting set by a constant factor, we can change the number of active vertices
at the end of the process drastically; in particular, the number of active vertices
at the end of the process is non-monotonic in the size of the starting set. 
We remark that the condition $\threshold\geq (\log n)^{2+\eps}$
is essentially best possible: for $\threshold\leq (\log n)^{2-\eps}$ the
statement of the theorem is not true, cf.\ the argument following
Theorem~\ref{thm:concentration} on page~\pageref{thm:concentration}. If one
weakens the conditions of (ii) by just requiring
$\start \ge (1+\eps)\threshold$ or $p\gg 1/n$, then it is still true that
the process 
is unstable, but one cannot predict where it ends.

%

A main feature of the classical bootstrap percolation processes is that
activation takes place in rounds. This phenomenon can be interpreted in the following way: when
a vertex turns active, then this information needs exactly one time unit to reach
its neighbors. In the second part of our paper we drop this assumption
and replace this synchronous model with an asynchronous one: each edge
independently draws a random transmission delay $\delta$  from an exponential
distribution with expectation one, and the information that the neighbor is
active requires time $\delta$ to travel from one vertex to another. The
activation rule itself remains unchanged: a vertex turns active as soon as it
is aware that $k$ of its neighbors are active (in the process without
inhibition), or as soon as it has notice of $k$ more active excitatory than
inhibitory neighbors (in the general case). Although the expected transmission
delay is one -- as it is deterministically in the synchronous model -- it turns
out that quantitatively and qualitatively the percolation process changes
rather dramatically.

\begin{theorem}\label{thm:async-percolation}
  Let $\eps, \tau, \gamma >0$, $k\ge 2$ and assume $1/n \ll p \ll n^{-1/k}$ and
  $\start \geq (1+\eps)\threshold(n,p,k,\tau)$.  Then there exists a constant
  $T=T(\eps,k)\geq 0$ such that the asynchronous bootstrap percolation process
  with  $\start \geq (1+\eps)\threshold$ and $n^{-1}\ll p\ll n^{-1/k}$
  a.a.s.\ satisfies the following.
  \begin{enumerate}[(i)]
    \item If $\tau <1/(1+\gamma)$, then the process almost percolates
      in time $T$. If, additionally, $p=\omega(\log n / n)$, then the process percolates
      completely within time $T$.
    \item If $\tau \geq 1/(1+\gamma)$ and if $\start = o(n)$, then a.a.s.~there are
      $(1-\tau)^kn/(\gamma\tau)^k + o(n)$ active vertices at time $T$. Also, the process stops
      with $(1-\tau)^k n/(\gamma\tau)^k + o(n)$ active vertices.
  \end{enumerate}
\end{theorem}

Note that this theorem implies two interesting phenomena. On the one hand, we see that
the asynchronous version accelerates the process dramatically: it essentially
stops after \emph{constant time}, as opposed to the roughly $\log_k\log (pn)$
rounds it takes in the synchronous model. 
On the other hand, we see that the final size of the process only depends on
the parameters $\tau$ and $\gamma$ but not on the size of the initial set, in
sharp contrast to the synchronous model. Therefore, by choosing the parameters $\tau$
and $\gamma$ appropriately, we can realize a normalization for an arbitrary
(linear) target size of the finally activated set.

We note that recently it was shown in~\cite{icalp} that  Theorem~\ref{thm:async-percolation} also implies similar results for other random graph models.

\subsection{Outline}

In~\cite{Janson2010} it is shown among other things that the classical
bootstrap percolation process on $G_{n,p}$ (without inhibition) consists of three phases.
While the active set is still very small (close to the threshold size
$\threshold$), the active set grows only by a small factor in each round. Once
the size of the active set is asymptotically larger than $\threshold$, the
growth of the active set picks up momentum, and we call this the \emph{explosion phase}.
Finally, once the active set has size at least $1/p$, the process
terminates in at most two more rounds (provided $p \gg \log n/n$):
one round to activate a linear subset of the vertices, and a possible
second round to to activate all remaining vertices. A similar situation occurs
also in the process with inhibition. In order to prove
Theorem~\ref{thm:synchron} we need to track the size of the active set very
precisely during the first two phases (as a function of the size of the starting set). In principle, this
seems like an easy task: given an active set of size $x$, we expect that in the
next round we activate $(n-x)\cdot \Pr[\Bin(x,p)\ge k]$ additional vertices. 
The problem is, of course, that such a simple approach ignores the dependencies between rounds.
We overcome this issue by defining a different probability space
(Section~\ref{sec:model}) that describes the same process but is more amenable
to a formal analysis. In this section we also prove some general properties of
the percolation process that apply both to the synchronous and the asynchronous
case. In Section~\ref{sec:introsync} we then use these prelimiaries to first
describe the evolution of the size of the active set as a function of the
number rounds very precisely (Theorem~\ref{thm:concentration}) and subsequently
use this result to prove Theorem~\ref{thm:synchron}. In Section~\ref{sec:async}
we then consider the asynchronous version of the process. As we will see, this
process behaves similarly as the synchronous version in the very early stage of
the process (while still close to the threshold) but then speeds up
considerably. Theorem~\ref{thm:async-percolation} then follows from the fact
that the sum of the incoming signals (positive minus negative ones) essentially
performs a random walk where the bias is a function of $\tau$ and $\gamma$.

\section{Preliminaries and Definitions}
\label{sec:model}

The aim of this section is to define a  general \emph{bootstrap percolation
process} which subsumes both the synchronous and the asynchronous case, and to
prove some basic properties of this process.

\subsection{Formal definition of the percolation process}\label{sec:probspace}


In this section we describe a version of the bootstrap percolation process that is particularly amenable to its analysis. We first activate the vertices in the starting set, assuming without loss of generality that this set consists of the vertices $1,\ldots,\start$. Then for each $s\ge 1$ we provide just enough information with the $s$-th active vertex to determine whether a $(s+1)$-st vertex is activated. Crucially, this information does \emph{not} require knowledge of the labels of the active vertices. In this way we can determine properties of active sets of a certain size, without actually knowing which vertices belong to this set. We now turn to the details.

Let $n\in \mathbb{N}$, let $p,\tau \in [0,1]$, and let $\gamma \in [0,1/p]$. Let $\Phi$
be a random variable taking values in the positive reals. Define the
product probability space
\[(\Omega, {\cal A}, \text{Pr}) = \prod_{1\le i,v\le n} \left( X_{iv}^+\times X_{iv}^- \times \Phi_{iv}\right)\;\times\; \prod_{i=1}^n \Psi_i\text, \]
where
\begin{itemize}
  \item $X_{iv}^+$ is the probability space of a Bernoulli random variable with parameter
    $p$,
  \item $X_{iv}^-$ is the probability space of a Bernoulli random variable with parameter
    $\gamma p$,
  \item $\Psi_i$ is the probability space with $\Pr[-1] = \tau$ and $\Pr[+1]=1-\tau$, and
  \item $\Phi_{iv}$ is the probability space on $\mathbb R_{>0}$ given by the distribution of $\Phi$.
\end{itemize}
By abuse of notation, we also denote the random variables corresponding to the
spaces $X_{iv}^+$, $X_{iv}^-$, $\Phi_{iv}$, and $\Psi_{i}$ again by $X_{iv}^+$,
$X_{iv}^-$, $\Phi_{iv}$, and $\Psi_{i}$, respectively. It will always be clear
from the context which interpretation we have in mind.

Before precisely defining the percolation process we give the intended
interpretations of these random variables. By symmetry, we assume that the
initially active set is $[\start]=\{1,\ldots,\start\}$. Define
\[ x_i := i\qquad\text{for all $1\le i\le \start$}\text. \]
In general, $x_i$ will be the label of the $i$-th vertex that becomes active in the percolation process, where, if several vertices should become active simultaneously 
ties are broken arbitrarily, for example by the natural ordering of the vertices. Then
\begin{itemize}
  \item $\Psi_{i}$ determines the sign of vertex $x_i$. That is, vertex $x_i$ is inhibitory if and only if $\Psi_{i}=-1$, which happens with probability $\tau$, and excitatory otherwise;
  \item $X_{iv}^-$ and $X_{iv}^+$ describe whether there is a directed edge from vertex $x_i$ to $v$:
    there is a directed edge from $x_i$ to $v$ exactly if
    either $\Psi_iX_{iv}^+ = 1$ or $\Psi_iX_{iv}^-=-1$. Note that the roles of
    $i$ and $v$ are not interchangeable: while $v$ represents a vertex of the underlying graph, $i$ represents the \emph{index} of the $i$-th vertex that becomes active.  
  \item $\Phi_{iv}$ describes the delay of the edge $(x_i,v)$. In the synchronous model, the delay is a constant of value $1$, while in the asynchronous model, it is an exponentially distributed random variable with parameter $1$. Note that for ease of analysis we define these random variables regardless of whether $X_{iv}=1$ or not.
\end{itemize}
For every $s\in [n]$, we define random variables
$E_s, I_s \colon \Omega \to \mathcal{P}([s])$ by 
\[ E_s := \{ i\in[s] \mid \Psi_i = +1\} \quad \text{and} \quad I_s := \{i\in[s] \mid \Psi_i = -1\}\text,\] respectively.
These are the sets containing the indices of the active excitatory resp.\ inhibitory vertices at the time at which exactly $s$ vertices are active.

We can now describe formally how elements $\omega \in \Omega$ define a \emph{percolation process}
\[((x_1,t_1),\dotsc,(x_n,t_n))\] with starting set $[\start]$ on $G_{n,p}$.
First, activate all the vertices in $[\start]$ at time $t=0$ by letting $x_s = s$ and $t_s=0$ for all $1\leq s \leq \start$. Assume now that active vertices $x_1,\ldots,x_s$ are given, where $s\ge \start$.  Also assume that for each such vertex $x_i$ we know the time $t_i$ when it turned active.

Then $x_{s+1}$ is defined as follows. First, for each vertex $v\in[n]\setminus\{x_1,\ldots,x_s\}$ we determine the earliest time $t_v^{(s)}$ at which $v$ has received $k$ more excitatory than inhibitory signals from the set $\{x_1,\ldots,x_s\}$:
\begin{multline*}
  t_{v}^{(s)} := \min\Big\{ t\in \mathbb R_{\ge 0}\mid |{\{i\in E_s\mid X_{iv}=1\text{ and } t_i + \Phi_{iv}\le t \}}| \ge \\
  k+|{\{i\in I_s \mid X_{iv}=1\text{ and } t_i + \Phi_{iv}\le t\}}|\Big\}\text,
\end{multline*}
where $\min{\emptyset}=\infty$. If there is some vertex $v$ for which $t^{(s)}_v<\infty$, then we define
\[ t_{s+1} := \min{\big\{ t_v^{(s)} \mid v \in[n]\setminus\{x_1,\dotsc,x_s\} \big\}}\text. \]
In this case, let $J:=\{ v\in [n]\setminus \{x_{1},\dotsc,x_s\} \mid t_v^{(s)}=t_{s+1}\}$ and $j:= |J|$. We set $t_{s+i} := t_{s+1}$ for $2\leq i\leq j$, and we let $x_{s+1} < \ldots < x_{s+j}$ be the (uniquely determined) vertices such that
\[ \{x_{s+1},\ldots,x_{s+j}\} = J\text.\]
If, on the other hand, we have
$t_v^{(s)}=\infty$ for all $v\in [n]\setminus\{x_1,\ldots,x_s\} $, then the process stops and we
set $t_{s'} := \infty$ and
$x_{s'} := \min{\{v\in [n]\setminus \{x_1,\ldots x_{s'-1}\}\}}$
for all $s' \ge s+1$,
i.e., we enumerate all remaining vertices by increasing label.

Finally, we introduce some more useful notation. For every $s\in [n]$ and $v\in [n]$, we define the random variables
\[ N_s^+(v) := \sum_{i\in E_s} X_{iv}^+\quad \text{and}\quad N_s^-(v):=\sum_{i\in I_s}X_{iv}^-\text.\]
Note that  $N^+_s(v)$ and $N^-_s(v)$ are the number of excitatory and inhibitory active neighbors of $v$ at the time at which exactly $s$ vertices are active. For brevity, we also use
$N_s(v):= N^+_s(v)+N^-_s(v)$.
\begin{remark}\label{remark:cond}
  From the definition of the probability space it follows immediately that
  for all positive integers $e$, the conditional distributions of $N^+_s(v)$ and $N^-_s(v)$
  given $|E_s|=e$ are binomial. More specifically, for every $0\leq x\leq s$, we have
  \[\Pr[N^+_s(v)=x\mid |E_s| = e] = \binom{e}{x}p^x (1-p)^{e-x}\]
  and
  \[\Pr[N^-_s(v)=x\mid |E_s| = e] = \binom{s-e}{x}(\gamma p)^x (1-\gamma p)^{s-e-x}\text.\]
  Also, for distinct vertices $v$ and $w$, the random variables
  $N^+_s(v)$, $N^-_s(v)$, $N^+_s(w)$ and $N^-_s(w)$ are mutually conditionally independent, given the value
  of $|E_s|$. In addition, note that $N_s(v)\sim\Bin(s,\hat p)$, for every $s,v\in [v]$,
  where $\hat p = (1-\tau)p+\tau \gamma p$.
\end{remark}


We will make frequent use of the following concentration bounds on the binomial distribution~\cite{RandomGraphsBook}.
\begin{lemma}[Chernoff]
  \label{lemma:Chernoff}
  Let $X_1,\ldots,X_n$ be independent Bernoulli variables with $\Pr[X_i = 1] = p$ and $\Pr[X_i = 0] = 1-p$ for all $1\leq i \leq n$, and let $X = \sum_{i=1}^n X_i$. Then for every $0 \leq \delta \leq 1$,
  \[
  \Pr[X \geq (1+\delta)np] \leq e^{-{\delta^2}np/3} \qquad \text{and}\qquad \Pr[X \leq (1-\delta)np] \leq e^{-{\delta^2}np/3}.
  \]
\end{lemma}

\subsection{General properties of the percolation process}\label{ssec:basic}

In this subsection we prove some properties of the probability space that are
independent of the distribution of the transmission delays $\Phi_{iv}$. These results thus apply
equally in the synchronous and the asynchronous case.

Let us start with the following simple fact, which states that at every point in time, the numbers of active
excitatory and inhibitory vertices are close to their expectations.

\begin{lemma}\label{lemma:basic1}
  Let $\delta_0=\delta_0(n)\in (0,1/2)$ be such that $\delta_0^2\start =\omega(
  -\log\delta_0)$.
  Then a.a.s.\ the percolation process satisfies
  \[\abs{E_s}\in (1\pm \delta_0)(1-\tau)s\quad\text{and}\quad\abs{I_s}\in (1\pm\delta_0)\tau s\]
  for all $s\ge \start$.
\end{lemma}
\begin{remark}
  We will apply this lemma in two settings: first, when $\delta_0$ is constant and $\start = \omega(1)$,
  second, when $\delta_0 = (\log n)^{-1-\eps/3}/(10k)$ and $\start \geq (\log n)^{2+\eps}$, for some constant $\eps>0$.
  Note that in both cases, the condition $\delta_0^2\start = \omega(-\log\delta_0)$ is satisfied.
\end{remark}
\begin{proof}[Proof of Lemma~\ref{lemma:basic1}]
  If $\tau =0$ or $\tau =1$ there is nothing to show. So assume $0<\tau<1$.
  It follows directly from the definitions that for every $s\in [n]$, we have $|E_s|\sim\Bin(s,1-\tau)$.
  Then Lemma~\ref{lemma:Chernoff} and the union bound imply that 
  \[\Pr[\exists s\geq \start : \abs{E_s} \not\in(1\pm\delta_0)(1-\tau)s] \le \sum_{s\ge \start}2e^{-\frac{\delta_0^2}3(1-\tau)s} = o(1),
  \]
  where we  used that
  $\delta_0^2 \start = \omega(-\log\delta_0)$ and routine calculations to obtain the last equality.
  The statement for $|I_s|$ is proved similarly.
\end{proof}

Our primary goal in this subsection is to introduce a general method to prove
that the process reaches a certain number of active vertices in a certain period of time.
To do this, for every $s\in [n]$ and $r\in \mathbb R_{>0}$, we define
\[ L_s(r) := |\{\start < v\leq n\mid N^+_s(v) = k \text{ and }N^-_{10s}(v)=0\text{ and }\max_{i\in E_s}X^+_{iv}\Phi_{iv}\leq r \}|\text. \]
The random variable $L_s(r)$ has the following very useful property: assume that exactly $s$ vertices are active at
some time $t$, and let $r$ be any positive real number; then at time $t+r$, there will be at least $\min{\{\start+L_s(r),10s\}}$ active vertices -- indeed, unless $10s$ vertices are activated before time $t+r$, every vertex counted in $L_s(r)$ will be active by time $t+r$. (Here the value $10s$ has no deeper meaning: we just need some value sufficiently larger than $s$.) Therefore, if we want to show that many vertices turn active quickly, then 
we need to prove lower bounds for the variables $L_s(r)$. This is what we will do in the next lemma.

For the analysis, it turns out to be very useful to parametrize the number of active vertices at a given
time as $s=x\cdot\threshold/(1-1/k)$, for some $x >0$. For this reason, we introduce the notation
\begin{equation}\label{eq:def:box} \Lambda =\Lambda(n,p,k,\tau):= \frac{\threshold}{1-1/k} = \left(\frac{(k-1)!}{(1-\tau)^knp^k}\right)^{1/(k-1)}\text,\end{equation}
and note that  $\Lambda$ satisfies
\begin{equation}\label{eq:prop:box} \frac{(1-\tau)^knp^k\Lambda^k}{(k-1)!}=\Lambda\text.\end{equation}

The following lemma shows essentially that, conditioned on the event that the values
$|E_s|$ are very close the their expectations, it is unlikely 
that there is some $\start \leq s\ll 1/p$ for which $L_s(r)$ is very small.
\begin{lemma}\label{lemma:lower}
  There exists a positive constant $c=c(\tau,k)$ such that if $\start\geq \threshold$ and $p\geq n^{-1}$, then
  the following holds for every $2k^2/((1-\tau)\start)\leq \delta_0\leq 1/(30k)$.
  Let $\eta\in [10k\delta_0,1/2)$ and $\delta = \eta/(10k)$.
  Write
  $\mathcal{E}$ for the event that $|E_s|\in (1\pm\delta_0)(1-\tau)s$ holds for all $s\geq \start$.
  Then for every $\start \leq s=x\thebox \leq \min\{\delta/p,\delta/(\gamma p)\}$ and $r\in \mathbb R_{>0}$, we have
  \[ \Pr[L_s(r) \geq (1- \eta)\Pr[\Phi\leq r]^kx^{k}\Lambda/k\mid \mathcal E] \geq 1-e^{-c\eta^2\Pr[\Phi\leq r]^kx^k\Lambda}\text.\]
\end{lemma}
\begin{proof}
  Write $\mathcal{E}_s(a)$ for the event that $|E_s|=a$, and $\mathcal I_{s}(b)$ for the event that
  $|I_{s}|=b$. Fix some $\start\leq s\leq \min\{\delta/p,\delta/(\gamma p)\}$ and $r\in \mathbb R_{>0}$.

  We  first prove that for all integers $(1-\tau)s/2\leq a \leq s$ and $0\leq b\leq 10s$ we have
  \begin{equation}\label{eq:lower}
    \EE[L_s(r)\mid\mathcal{E}_s(a)\cap \mathcal I_{10s}(b)]\geq (1-\delta)^{13}\cdot \frac{n a^kp^k}{k!}\cdot \Pr[\Phi\leq r]^k\text.
  \end{equation}
  To see this, fix a vertex $\start<v\leq n$ arbitrarily and write $\mathcal L(v)$ for the event
  that $N^+_s(v)= k$ and $N^-_{10s}(v)= 0$ and $\max_{i\in E_s}X_{iv}\Phi_{iv}\leq r$. 
  Using the conditional independence of $N^+_s(v)$ and $N^-_s(v)$ (see Remark~\ref{remark:cond}), and the independence
  of the variables $\Phi_{iv}$, we get
  \begin{align*}
    \Pr[\mathcal L(v)\mid \mathcal{E}_s(a)\cap \mathcal I_{10s}(b)]
    &= \Pr[N^+_s(v)=k \mid \mathcal{E}_s(a)]\cdot \Pr[N^-_{10s}(v)=0\mid \mathcal{I}_{10s}(b)]\cdot \Pr[\Phi\leq r]^k\\
    &= \binom{a}k p^k(1-p)^{a-k}\cdot(1-\gamma p)^{b}\cdot\Pr[\Phi\leq r]^k\\
    &\ge \left(1-\frac{k^2}{a}\right)\frac{a^kp^k}{k!}(1-p)^s\cdot (1- \gamma p)^{10s}\cdot \Pr[\Phi\leq r]^k\\
    &\ge  (1-\delta)^{12}\cdot\frac{a^kp^k}{k!}\cdot \Pr[\Phi\leq r]^k,
  \end{align*}
  using $s\leq \min{\{\delta/p,\delta/(\gamma p)\}}$,
  and the fact that  $k^2/a\leq 2k^2/((1-\tau)\start)\leq \delta_0\leq \delta$.

  By definition, we have $L_s(r) = |\{\start<v\leq n\mid \mathcal L(v)\}|$, and thus
  \begin{align*}\EE[L_s(r) \mid \mathcal{E}_s(a)\cap \mathcal I_{10s}(b)] & \geq (n-\start)(1-\delta)^{12}\cdot \frac{a^kp^k}{k!}\cdot \Pr[\Phi\leq r]^k\\
    & \geq (1-\delta)^{13}\cdot \frac{n a^kp^k}{k!}\cdot \Pr[\Phi\leq r]^k\text,
  \end{align*}
  for all $a$ and $b$ as above, proving \eqref{eq:lower}. Here we used that $\start\leq \delta/p\leq \delta n$.

  Now, observe that, by definition of the underlying probability space, the events $\{\mathcal L(v)\mid v\in [n]\}$ are conditionally independent given $\mathcal E_s(a) \cap \mathcal I_{10s}(b)$, for all choices of $a$ and $b$.
  Then, by Lemma~\ref{lemma:Chernoff}, writing $\mu_{a,b} := \EE[L_s(r) \mid \mathcal{E}_s(a)\cap \mathcal I_{10s}(b)]$, we have 
  \[  \Pr[L_s(r) < (1-\delta) \mu_{a,b}\mid \mathcal{E}_s(a)\cap \mathcal{I}_{10s}(b)]
  <e^{-\delta^2\mu_{a,b}/3}
  <e^{-\delta^2\Pr[\Phi\leq r]^k(1-\delta)^{13}na^kp^k/(3k!)}. \]
  for all $(1-\tau)s/2\leq a\leq s$ and $0\leq b\leq 10s$.
  If we condition on the event $\mathcal E$, then we may assume $a\in (1\pm \delta)(1-\tau)s$, and we get
  \begin{align*}
    \Pr[L_s(r) &< (1-\delta)^{14+k}\frac{\Pr[\Phi\leq r]^k(1-\tau)^kns^kp^k}{k!} \mid \mathcal E]\\
    &< e^{-\delta^2\Pr[\Phi\leq r]^k(1-\delta)^{14+k}(1-\tau)^kns^kp^k/(3k!)}. \end{align*}
  The lemma now follows  using $s=x\Lambda$ with \eqref{eq:prop:box}
  and from $\delta=\eta/(10k)$, which implies that $1/2\le1-\eta\leq(1-\delta)^{14+k}$ holds for $k\geq 2$. 
\end{proof}

\begin{remark}\label{rem:lsrexp}
  For later reference, we just note here that \eqref{eq:lower} in the proof above, together with
  \eqref{eq:prop:box}, implies that for every
  $\start\leq s=x\thebox\leq \min\{1/(10kp),1/(10k\gamma p)\}$ and $(1-\tau)s/2\leq a \leq s$, we have
  \[\EE[L_s(r)\mid |E_s|=a] = \Omega(x^k\Lambda\Pr[\Phi\leq r]^k)\text.\]
\end{remark}

Recall that, by the definition of $L_s(r)$, if there are exactly $s$ active vertices at time $t$, then at time $t+r$ there are at least $\min\{\start + L_s(r), 10s\}$ active vertices. We now use this observation to obtain a lower bound on the growth of the process.

\begin{corollary}\label{cor:growth_lower}
  For every $\eps>0$, there exist positive constants $c_0=c_0(\eps,k,\Phi)$ and $\delta=\delta(\eps,\gamma,k)$ such that
  for every function $f\colon \mathbb N\to \mathbb N$, the percolation process with
  $\start \geq (1+\eps)\threshold$ and $n^{-1}\leq p\ll n^{-1/k}$ is a.a.s.~such that
  at least $\min\{f(n)\start,\delta/p\}$ vertices are active at time $c_0\log(f(n))$.
\end{corollary}
\begin{proof}
  Fix a sufficiently small constant $\eta=\eta(\eps,k)\in (0,1)$ and let $\delta:=\eta/(10k)$.
  Note that $p\ll n^{-1/k}$ implies that $\start \geq \threshold=\omega(1)$, so that for large enough $n$,
  we have $\delta \geq 2k^2/((1-\tau)\start)$. Also, since $\start=\omega(1)$,
  by Lemma~\ref{lemma:basic1}, we may assume
  that $|E_s|\in (1\pm\delta)(1-\tau)s$ holds for all $s\geq \start$.

  Choose $r=r(k,\Phi)$ to be so large that $\Pr[\Phi\leq r]^k\geq 1-\eta$.
  Applying Lemma~\ref{lemma:lower}, we get
  that for every $\start \leq s=x\Lambda\leq \min\{\delta/p,\delta/(\gamma p)\}$, we have
  \[ \Pr[L_s(r) \geq (1-\eta)^2x^{k}\Lambda/k] \geq 1-e^{-\Omega(x^k\Lambda)}\geq 1-e^{-cs}\text, \]
  for some positive constant $c=c(\eta,k,\tau)$. 
  Since $\sum_{s\geq \start}e^{-c s} =
  e^{-c\start}/(1-e^{-c}) = o(1)$
  the process is such that a.a.s.
  \[ L_s(r) \geq (1-\eta)^2x^{k-1}s/k\]
  holds for all $\start\leq s\leq \min{\{\delta/p,\delta/(\gamma p)\}}$. Thus
  \[ \frac{\start + L_s(r)}{s}\geq (1-\eta)^2\frac{\start/\Lambda + x^k/k}{x}\geq (1-\eta)^2(\start/\threshold)^{\frac{k-1}{k}}\text,\]
  where the last inequality follows by minimizing over $x\ge \start/\Lambda = (1-1/k)(\start/\threshold)$ and the minimum is obtained at $x^k=\start/\threshold$.
  As $\start\geq (1+\eps)\threshold$, the definition of $L_s(r)$ implies that
  after every time period of constant length  $r$, 
  the number of active vertices is multiplied by a constant factor of size at least
  $\min{\{(1-\eta)^2(1+\eps)^{k-1},10\}}$.
  If $\eta$ is small enough, then $(1-\eta)^2(1+\eps)^{k-1}>1$, and the corollary follows.
\end{proof}

Corollary~\ref{cor:growth_lower} implies that the process grows at least at an exponential rate. 
In fact, it will turn out that in both the
synchronous and the asynchronous case, the growth is actually much faster. However, the corollary already
implies two useful facts. Firstly,
regardless of the fraction $\tau n$ of inhibitory vertices, the process will always reach $\Theta(1/p)$ active vertices.
Secondly, in order to activate 
a constant multiple of the starting set we only need $O(1)$ time. 

\begin{corollary}\label{cor:basic}
  For every  $\eps >0$ there exists a $\delta=\delta(\eps,\gamma,k)>0$ such that for $\start \geq (1+\eps)\threshold$ and $n^{-1}\leq p\ll n^{-1/k}$, the process a.a.s.~activates at least $\delta/p$ vertices.
\end{corollary}

\begin{corollary}\label{cor:babysteps}
  For every $\eps>0$ and $c>0$ there exist constants $T=T(\eps,k,c,\Phi)>0$ and $\delta=\delta(\eps,\gamma,k)>0$
  such that if $\start \geq (1+\eps)\threshold$
  and $n^{-1}\leq p\ll n^{-1/k}$, then the percolation
  process a.a.s.~activates at least $\min\{c\start,\delta/p\}$ vertices in time $T$.
\end{corollary}

\subsection{Phases of percolation}

It is interesting to note that in the statements of Corollaries~\ref{cor:basic} and \ref{cor:babysteps},
the inhibition parameter $\tau$ is not mentioned at all. The reason for this is that, as long as there are $o(1/p)$
active vertices,
the number of vertices that have even one active inhibitory neighbor is $o(n)$; in this sense, the behavior of the
process is almost completely unaffected by the presence of inhibitory vertices until there are $\Omega(1/p)$
active vertices.

Thus, the evolution of the percolation process divides naturally into two separate phases: the \emph{initial phase} $\start \leq s \ll 1/p$,
during which the growth is largely unaffected by inhibition, and the \emph{end phase} $s=\Omega(1/p)$, where many vertices start to have inhibitory neighbors.

If $\start =\Theta(\threshold)$, then one can further subdivide the initial phase into two phases with
$s= \Theta(\threshold)$
and $\threshold \ll s\ll 1/p$, respectively. The former is called the
\emph{startup phase}, and Corollary~\ref{cor:babysteps} shows that if $\start
\geq (1+\eps)\threshold$, then the time spent in the startup phase is bounded
from above by some constant. However, one can show that if
$\start=\Theta(\threshold)$, then this upper bound is close to the truth,
i.e., the size of the active set really increases only by some (small) constant factor
in each round. In contrast, once we have $s\gg \threshold$, the rate of growth
speeds up considerably. Thus we call this second phase the \emph{explosion
phase}. As we will see, the time that is spent in the explosion phase depends
significantly on the distribution $\Phi$ of the signal delays: for the
synchronous process (where $\Phi$ is identically one), the time spent in the
explosion phase is $\log_k\log_{(\start/\threshold)} (pn)+O(1)$, while for the
asynchronous process (where $\Phi$ is exponentially distributed with mean one),
it is $o(1)$.

\section{Synchronous Bootstrap Percolation}\label{sec:introsync}

In this section we study the synchronous bootstrap percolation process with inhibition. Recall that the synchronous
process is defined by taking all edge delays to be constants $\Phi_{iv}=1$. Then it is clear that for every vertex $x_i$, the time $t_i$ at which $x_i$ becomes active is either a non-negative integer or $\infty$. For this reason, we can view the percolation process as happening in discrete rounds $t=0,1,2,\dotsc$.
We  write
\[ a_t := |\{ x_i \mid i\in [n] \text{ and } t_i \le t \}| \]
for the number of vertices that are active after round $t$, and
\[a^* := \max{\{a_t\mid t\geq 0\}}\]
for the number of vertices at termination.

For $\tau=0$ (the $G_{n,p}$ case without inhibition),
the process was analyzed in great detail in~\cite{Janson2010}. Among other results, it was shown that
$\threshold(n,p,k,0)$ is the threshold for percolation in $G_{n,p}$, and moreover
that the process with $\start \leq (1-\eps)\threshold$ will a.a.s.~not even activate more than $k\start/(k-1)$
vertices. Moreover, the authors of~\cite{Janson2010} determined the typical number of rounds until percolation up to an additive constant.

In the case with inhibition, it is not clear that we percolate to a point where all (or at least most of) the excitatory vertices are active. Corollary~\ref{cor:basic} guarantees that inhibition essentially plays no role while we have at most $\delta/p$ active vertices, but from then on things may change. Our plan for the rest of this section is as follows. First we  show that we can describe the dynamics of the percolation process very precisely up to $\delta /p$ active vertices. Then we  show that this implies that the process with inhibition actually follows a complicated pattern, where the number of finally active vertices depends on the size on the starting set in a {\em non-monotone} way. We start by proving a concentration theorem.

\begin{theorem}\label{thm:concentration}
  For every $\eps > 0$ there exists $\delta=\delta(\eps,\gamma,k,\tau) > 0$ such that,
  for the sequence $(\hat a_t)_{t\geq 0}$ defined by  
  \begin{equation}\label{eq:defahat}
    \hat a_0 := \start \qquad\text{and}\qquad\hat{a}_{t+1} := \hat{a}_0 + (1-\tau)^knp^k\frac{\hat{a}_t^k}{k!}\text,
  \end{equation}
  the synchronous process
  with $\start\geq \max{\{(1+\eps)\threshold,(\log n)^{2+\eps}\}}$ and $p \gg n^{-1}$ a.a.s.~satisfies
  \[ (1-\eps)\hat{a}_t\leq a_t \leq (1+\eps)\hat{a}_t \]
  for all $t\ge 0$ such that $\hat{a}_t \leq \delta n$.
\end{theorem}

One can show that the requirement $\start \geq (\log n)^{2+\eps}$ is tight in the following sense: if we have $\start < (\log n)^{2-\eps}$ for some constant $\eps>0$, then with non-negligible probability the number of active vertices after the first round will deviate from its expectation by a factor that, accumulated over many rounds, makes it impossible for such a statement to hold.
More precisely, assume that $\start = (1+\eps)\threshold \leq (\log n)^{2-\eps}$, then 
the expectation of $a_1$ is
\[\mathbb E[a_1] \approx \start+(1-\tau)^k np^k
\start^k/k!\stackrel{\eqref{eq:prop:box}}=  \Theta(\start).\] Let $\delta=
(\log n)^{\eps/3-1}\ll  \start^{-1/2}$. By the tightness of the Chernoff bound
(or by normal approximation), the probability that $a_1 > (1+\delta)\mathbb
E[a_1]$ is at least some constant.
By the definition of the sequence $\hat a_i$,
the factor $(1+\delta)$ will blow up at a doubly exponential rate,
and after $i$ rounds, the uncertainty on $a_{i}$ will
be $(1+\delta)^{\Theta(k^{i})}$. We will see
(cf.~Lemma~\ref{lemma:time-to-1/p}) that the number of rounds with $\hat
a_t\leq \delta n$ is $\ell=\log_k\log (n)-O(1)$. So the uncertainty after
$\ell$ rounds would be $(1+(\log n)^{\eps/2-1})^{\Theta(\log n)}\gg 1$, which
shows that it is impossible for $a_{\ell}$ to be concentrated around $\hat
a_\ell$.

%

\subsection{The speed of round-based percolation}

In Subsection~\ref{ssec:basic}, we introduced a general approach for proving that the percolation
progresses grows at least with a certain speed: if, at some point, there are $s$ active vertices, then
after waiting for a time period of length $r$, there will be at least $\min{\{\start+L_s(r),10s\}}$ active vertices.
In the case of synchronous percolation, we can strengthen (and simplify) this statement a bit. Define, for
every $s\in [n]$,
\[ L_s := |\{\start < v\leq n\mid N^+_s(v) = k \text{ and }N^-_{s}(v)=0|\text.\]
Note that in comparison to the definition of $L_s(r)$, we replaced the condition $N^-_{10s}(v)=0$ by $N^-_{s}(v)=0$ and omitted the condition on the random variables $\Phi_{iv}$. Nevertheless, due to the round-based nature of the synchronous process, we still can conclude:
if there are $s=a_t$ active vertices at time $t$, then at time $t+1$, there
will be at least $\start+L_s$ active vertices.

To prove concentration of the sequence $(a_t)_{t\geq 0}$,
we need to show that this lower bound for $a_{t+1}$ is more or less tight. To do this, we introduce
a second set of random variables. For every $s\in [n]$, define
\[ U_s := |\{\start < v\leq n\mid N^+_s(v) \geq k\}|\text.\]
With this definition, it is clear that if at some time $t$, there are $s=a_t$ active vertices,
then at time $t+1$, there will not be more than $\start+U_s$ active vertices. The next lemma
says that for all $s\ll 1/p$, the upper and lower bounds $U_s$ and $L_s$ are not likely to differ by much.

\begin{lemma}\label{lemma:basic2}
  There exists a positive constant $c=c(\tau,k)$ such that if $\start\geq \threshold$ and $p\geq n^{-1}$, then
  the following holds for every $2k^2/((1-\tau)\start)\leq \delta_0\leq 1/(30k)$.
  Let $\eta\in [10k\delta_0,1/2)$ and $\delta = \eta/(10k)$.
  Write
  $\mathcal{E}$ for the event that $|E_s|\in (1\pm\delta_0)(1-\tau)s$ holds for all $s\geq \start$.
  Then, for every $\start \leq s=x\Lambda\leq \min{\{\delta/p,\delta/(\gamma p)\}}$, we have
  \[ \Pr[L_s \geq (1- \eta)x^k\Lambda/k \mid\mathcal E] \geq 1-e^{-c\eta^2x^k\Lambda}\]
  and
  \[ \Pr[U_s \leq (1+ \eta)x^k\Lambda/k\mid \mathcal E] \geq 1-e^{-c\eta^2x^k\Lambda}\text.\]
\end{lemma}
\begin{proof}
  Fix some $\start\leq s=x\Lambda\leq \delta/p$.
  Since $L_s\geq L_s(1)$, the statement for $L_s$ follows directly from Lemma~\ref{lemma:lower}. 
  For the statement for $U_s$, given $0\leq a\leq s$,
  write $\mathcal{E}_s(a)$ for the event that $|E_s|=a$. By Remark~\ref{remark:cond} we know that, conditioned on $\mathcal E_s(a)$, the variable $U_s$
  follows a binomial distribution.
  In order to obtain an upper bound on $\mu_a := \EE[U_s \mid \mathcal E_s(a)]$ we use 
  the following property of the binomial distribution: if $W\sim \Bin(n,p)$
  with $np \le 1/2$, then we have (see for example~\cite{Bahadur1960})
  \begin{equation}\label{eq:binombound}
    \Pr[W \geq b] \leq (1+ 2np)\Pr[W =b]\qquad\forall b\geq 0\text.
  \end{equation}
  As  $a\leq s\leq \delta/p$ this bound implies that
  \begin{align*}
    \mu_a = (n-\start)\cdot \Pr[\Bin(a,p)\geq k]  \leq n\cdot (1+2pa)\cdot \frac{a^kp^k}{k!}\;\le \;
    (1+\delta)^{2}\cdot \frac{na^kp^k}{k!}\text.
  \end{align*}
  From here an application of the Chernoff bound (Lemma~\ref{lemma:Chernoff}) gives that
  for every $(1-\tau)s/2 \leq a \leq s$, we have
  \[\Pr[U_s >(1+\delta)^3 na^kp^k/k! \mid \mathcal E_s(a)]<e^{-\delta^2\mu_a/3}=
  e^{-\delta^2\Omega(x^k\Lambda)}\text,\]
  since $L_s\leq U_s$ implies that $\mu_a\geq \EE[L_s\mid \mathcal E_s(a)]=\Omega(x^k\Lambda)$, by Remark~\ref{rem:lsrexp}.
  Recall that the statement we want to prove conditions on the event $\mathcal E$, meaning that we can assume $a\in(1\pm\delta)(1-\tau)s$. The above bound thus implies
  \[\Pr[U_s >(1+\delta)^{3+k} n(1-\tau)^ks^kp^k/k! \mid \mathcal E]<
  e^{-\delta^2\Omega(x^k\Lambda)}\text,\]
  We have
  $(1+\delta)^{k+3}=(1+\eta/(10k))^{k+3}\leq 1+\eta$, for all $\eta\in (0,1)$ and $k\geq 1$.
  Then the lemma follows with an application of \eqref{eq:prop:box}.
\end{proof}

\subsection{The expected trajectory $(\hat{a}_t)_{t\geq 0}$}

Lemma~\ref{lemma:basic2} tells us that if there are $a_t=x\Lambda$ active vertices in round $t$, then in round $t+1$,
there will be
\[a_{t+1}\approx\start+\frac{x^k\thebox}{k}
=\start + (1-\tau)^knp^k\frac{a_t^k}{k!} \] active vertices, using \eqref{eq:prop:box}.
This motivates the definition of a sequence $(\hat a_t)_{t\geq 0}$ in equation \eqref{eq:defahat} in
Theorem~\ref{thm:concentration}.
Note that if we parametrize $\hat{a}_t = x \thebox$, we get
\begin{equation}\label{eq:growth}
  \hat{a}_{t+1} = \start+\frac{x^k}{k}\thebox\text.
\end{equation}

In the next lemma we establish a simple fact on the minimal growth of the sequence $(\hat{a}_t)_{t\geq 0}$. 
\begin{lemma}
  \label{lemma:ahatgrowth}
  For all $t\geq 0$, we have
  $\hat{a}_{t+1}/\hat{a}_t\geq (\start/\threshold)^{\frac{k-1}{k}}$.
\end{lemma}
\begin{proof}
  Write $\hat{a}_t = x \thebox$.
  Then we obtain from \eqref{eq:growth} that
  \[ \hat{a}_{t+1}/\hat{a}_{t} = \frac{\start/\thebox + x^k/k}{x}\text. \]
  The minimum of this expression is achieved for $x = (\start/\threshold)^{1/k}$, where its value is
  $(\start/\threshold)^{\frac{k-1}k}$, completing the proof.
\end{proof}

The bounds from the previous lemma are weak, but nevertheless best possible:
the sequence $(\hat a_t)$ grows very slowly at the beginning. 
Once, however $\hat a_t$ is above, say, $2\thebox$, a doubly exponential growth kicks in, and implies
that the total number of rounds of the process is just doubly logarithmic, as our next lemma shows.

\begin{lemma}\label{lemma:time-to-1/p}
  For every $\eps>0$, there exists a constant $K=K(k,\eps)$ such that for all large enough $n\in \mathbb N$,
  the following holds, provided $\start\geq (1+\eps)\threshold$ and $p=\omega(n^{-1})$:
  \begin{enumerate}[(i)]
    \item $\hat{a}_t\geq n$ for all $t\geq \log_k\log_{(\start/\threshold)}(pn)+K$, and
    \item $\hat{a}_t\leq 1/p$ for all $t\leq \log_k\log_{(\start/\threshold)}(pn)-K$.
  \end{enumerate}
\end{lemma}
\begin{proof}
  First observe that by Lemma~\ref{lemma:ahatgrowth}, there exists a constant $t_0=t_0(\eps,k)$ such that
  $\hat a_{t_0} \geq (\start k/\threshold)\thebox$. By \eqref{eq:growth}, we see in particular that
  $\hat a_t= x\thebox$ implies $\hat a_{t+1}\geq (x^k/k)\thebox$.
  Using induction we get that for all $t\ge 0$, we have
  \[ \hat{a}_{t_0+t} \geq
  \frac{(\start k/\threshold)^{k^t} \thebox}{k^{1+k+k^2+\dotsb+k^{t-1}}}
  \geq \left(\frac{\start k}{\threshold k}\right)^{k^t}\thebox
  =(\start/\threshold)^{k^t}\thebox\text.\]
  It follows that for all $t\geq \log_k\log_{(\start/\threshold)}(1/(p \thebox))$, we have
  $\hat{a}_{t_0+t} \geq 1/p$, and so 
  \[ \hat a_{t_0+t+1}\ge\start + (1-\tau)^kn/k! = \Omega(n)\text,\]
  whence $\hat a_{t_0+t+2} = \omega(n)$, using $p=\omega(n^{-1})$.
  Since $\log_k\log_{(\start/\threshold)}(1/(p \thebox))$
  is within a constant difference of
  $\log_k\log_{(\start/\threshold)}(pn)$,
  this proves (i).

  For (ii), we may assume, again by Lemma~\ref{lemma:ahatgrowth}, that there is some smallest constant
  $t_0\geq 0$ such that $2\Lambda \leq \hat a_{t_0}\leq 1/p$. Now if $\hat a_{t}=x\Lambda\geq \start$
  for some $x\geq 2$, then, using \eqref{eq:growth} and $k\geq 2$, we have
  \[ \hat a_{t+1}= \start + \frac{x^k\thebox}{k}\leq (x+x^k/k)\thebox\leq x^k\thebox\text. \]
  By induction, we thus have
  \[ \hat a_t\leq\hat a_{t_0+t} \leq (\hat a_{t_0}/\thebox)^{k^t}\thebox
  \leq (\hat a_{t_0}/\threshold)^{k^t}\thebox\]
  for all $t\geq 0$, and it follows that for all
  $t\leq \log_k\log_{(\hat a_{t_0}/\threshold)}(1/(p \thebox))$, we have
  $\hat{a}_{t} \leq 1/p$.
  If $t_0=0$, then $\hat a_{t_0}=\start$. If $t_0>0$, then
  $\hat a_{t_0}/\threshold = O(1)$.
  In both cases, (ii) follows easily.
\end{proof}

\subsection{Initial phases -- proof of Theorem~\ref{thm:concentration}}

Assume that $\start \geq (1+\eps)\threshold$ holds for some constant $\eps>0$.
We want to show that a.a.s.,
\[ (1-\eps)\hat a_t \leq a_t\le (1+\eps)\hat a_t\]
holds for all $t\geq 0$ such that $\hat a_t \le \delta n$, where
$\delta=\delta(\eps,\gamma,k,\tau)$ is some positive constant.
The idea is to proceed by induction over $t$. Recall that for $t=0$ we have $\hat a_0 = a_0=\start$ by definition, so the base case is settled. The difficulty in the induction step is that
from one round to the next the error bounds that we can prove will worsen. Therefore, instead of showing $a_t \in (1\pm \eps) \hat a_t$, we need to show $a_t \in (1\pm \eps_t) \hat a_t$ for an appropriate sequence $(\eps_t)_{t\geq 0}$. Here is how we choose this sequence:
set $\eta_0 = (\log n)^{-1-\eps/3}$ and define $\eta_t$ for $t\geq 1$ by
\[1+\eta_t := (1+\eta_0)(1+ 20k\hat a_tp\cdot \max{\{1,\gamma\}}) \le 1+ \eta_0 + 40k\hat a_tp\cdot \max{\{1,\gamma\}}\text.\]
Finally, define the sequence $(\eps_t)_{t\geq 0}$ recursively by 
$$\eps_0 := 0\qquad\text{and}\qquad
1+\eps_t:= (1+\eta_{t-1})\cdot (1+\eps_{t-1})^k = \prod_{i=0}^{t-1}(1+\eta_i)^{k^{t-1-i}}\quad\text{for }t\ge 1,$$
where the last equality follows from a straightforward induction. 
Recall that we assume that $\start \geq (\log n)^{2+\eps}$, so that we have in particular that
$\eta_0^2\start \geq (\log n)^{\eps/3} = \omega(-\log{\eta_0})$.

\begin{lemma}\label{l:thm1:defs}
  For every $\eps>0$, there exists $\delta=\delta(\eps,\gamma,k,\tau)>0$ such that the following holds, assuming
  that $\start\geq (1+\eps)\threshold$ and $p=\omega(n^{-1})$.
  Define $(\eps_t)_{t\geq 0}$ and $(\eta_t)_{t\geq 0}$  as above and
  write $\ell$ for the largest positive integer such that $\hat a_\ell\leq \delta n$.
  Then 
  \begin{enumerate}[(i)]
    \item $\eps_t\leq \eps$ for all $0\leq t \leq \ell$,
    \item  $\ell= O(\log \log n)$, and
    \item $\eta_t\cdot \min{\{1,\gamma^{-1}\}}\geq (1+\eps_t)10k\hat a_tp$ holds for all $0\leq t \leq \ell$.
  \end{enumerate}
\end{lemma}

We defer the technical proof of this lemma to the end of this subsection and first show how it can be used in order to complete the proof of Theorem~\ref{thm:concentration}.
\begin{proof}[Proof of Theorem~\ref{thm:concentration}]
  Assume the sequences $(\eta_t)_{t\geq 0}$ and $(\eps_t)_{t\geq 0}$ are defined as above.
  As in the statement of Lemma~\ref{l:thm1:defs}, we define $\ell$ to be the largest positive integer
  $t$ for which $\hat a_t \leq \delta n$, for some sufficiently small positive constant $\delta=\delta(\eps,\gamma,k,\tau)$. By Lemma~\ref{l:thm1:defs}
  (ii) and since $\eta_0^2\start \geq (\log n)^{\eps/3}$, we know in particular that $\ell e^{-c\eta_0^2\start}=o(1)$ for any constant $c>0$.
  For every $i\geq 0$, let $\delta_i = \eta_i/(10k)$. Write $\mathcal E$ for the event that
  $|E_s|\in (1\pm \delta_0)(1-\tau)s$
  holds for all $s\geq \start$.

  Let $x_i := \hat a_i/\Lambda$ and let $s_i^{(1)} := (1-\varepsilon_i)\hat
  a_i$ and $s_i^{(2)} := (1+\varepsilon_i)\hat a_i$.
  Observe that by \eqref{eq:growth} and Lemma~\ref{lemma:ahatgrowth}, we have
  \[  x_i^k\Lambda/k = \hat a_{i+1}-\start \geq \varepsilon \start. \]
  Note that  Lemma~\ref{l:thm1:defs} (iii) implies that $s_i^{(1)},s_i^{(2)} \le
  \min{\{\delta_i/p,\delta_i/(\gamma p)\}}$. One easily checks that the other conditions
  of Lemma~\ref{lemma:basic2} are met, so we obtain that there is a constant $c >0$ such that
  \[ \Pr[L_{s_i^{(1)}} \geq (1-\eta_i)(1-\varepsilon_i)^k(\hat a_{i+1}-\start) \mid \mathcal E]
  \geq 1- e^{-c\eta_0\start} \]
  and
  \[ \Pr[U_{s_i^{(2)}} \leq (1+\eta_i)(1+\varepsilon_i)^k(\hat a_{i+1}-\start) \mid \mathcal E]
  \geq 1- e^{-c\eta_0\start}. \]
  We have $(1+\eta_i)(1+\varepsilon_i)^k= 1+\varepsilon_{i+1}$
  and one can see that this implies that
  $(1-\eta_i)(1-\eps_i)^k \geq 1-\varepsilon_{i+1}$.
  Moreover, by Lemma~\ref{lemma:basic1} and the fact that $\delta_0^2\start =
  \omega(-\log{\delta_0})$, we have $\Pr[\mathcal E]= 1-o(1)$.
  Thus, by the union bound, with probability $1-2\ell e^{-c\eta_0\start} = 1-o(1)$, we have
  \begin{equation}\label{eq:lsus}
    L_{s_i^{(1)}} \geq (1-\varepsilon_{i+1})(\hat a_{i+1}-\start)\quad\text{ and }\quad
    U_{s_i^{(2)}} \leq (1+\varepsilon_{i+1})(\hat a_{i+1}-\start)
  \end{equation}
  for all $0 \leq i < \ell$. In the following, we assume that this is the case.

  We now prove by induction that for each $0\leq i \leq \ell$, we have
  \begin{equation}\label{eq:conc}
    (1-\varepsilon_i) \hat a_i \leq a_i \leq (1+\varepsilon_i)\hat a_i.
  \end{equation}
  Note that by Lemma~\ref{l:thm1:defs} (i), this will complete the proof.
  Since $a_0 = \hat a_0$, Equation~\eqref{eq:conc} holds trivially for $i=0$. For the induction,
  assume that it holds for a given $i\geq 0$, that is, assume $s_i^{(1)} \leq a_i \leq s_i^{(2)}$.
  Then by the definition of the sets $L_{s_i^{(1)}}$ and $U_{s_i^{(2)}}$, we have
  \[\start + L_{s_i^{(1)}} \leq a_{i+1} \leq \start + U_{s_i^{(2)}}. \]
  By~\eqref{eq:lsus}, this implies
  \[ (1-\varepsilon_{i+1})\hat a_{i+1} \leq a_{i+1} \leq (1+\varepsilon_{i+1})\hat a_{i+1},\]
  completing the proof.
\end{proof}

\begin{proof}[Proof of Lemma~\ref{l:thm1:defs}]
  By Lemma~\ref{lemma:time-to-1/p}, we know that $\ell\le \log_k\log_{(\start/\threshold)}(np) +K$,
  for some constant $K=K(k,\eps)$, so (ii) is immediate.

  To prove (i), fix some $0\leq t\leq \ell$.
  Using the fact that $\log(1+x)\leq x$ holds for all $x>-1$, we get
  \[\log(1+\eps_t) = \sum_{i=0}^{t-1}k^{t-1-i}\log(1+\eta_i) \leq \eta_0\sum_{i=0}^{t-1}k^{t-1-i} + \max{\{1,\gamma\}}\cdot 40kp\sum_{i=0}^{t-1}k^{t-1-i}\hat a_i\text.\]
  We bound the two terms individually. Since $\start\geq (1+\eps)\threshold$, we have
  \[\eta_0 \sum_{i=0}^{t-1}k^{t-1-i} = \eta_0 \frac{k^t-1}{k-1}\leq \eta_0 \frac{k^{\ell}-1}{k-1}\leq
  \frac{k^{K(k,\eps)}\log_{(\start/\threshold)}(np)-1}{(\log{n})^{1+\eps/3}(k-1)} = o(1)\text.\]
  Now consider the smallest integer $t_0\geq 0$ such that $\hat a_{t_0} > 4k^3\Lambda$. By Lemma~\ref{lemma:ahatgrowth},
  we know that $t_0$ is bounded. Thus, using the upper bound on $t$, \eqref{eq:def:box} and $pn =\omega(1)$ we obtain
  \[\max{\{1,\gamma\}}\cdot 40kp\sum_{i=0}^{t_0-1}k^{t-1-i}\hat a_i = \Theta(p \Lambda k^{t}) = O(\log(pn) (pn)^{-1/(k-1)}) = o(1) \]
  and it thus remains to bound the quantity $\max{\{1,\gamma\}}\cdot40kp\sum_{i=t_0}^{t-1}k^{t-1-i}\hat a_i$.

  By~\eqref{eq:growth} we have ${\hat a_{i+1}}/{\hat a_i}\geq
  (4k^3)^{k-1}/k\geq 2k$ for every $i\geq t_0$. By induction, it follows
  that for every $t_0\leq i< t-1$, we have $\hat a_{t-1}\geq
  (2k)^{t-1-i}\hat a_{i}$. Moreover, by the definition \eqref{eq:defahat}
  of $\hat a_{\ell}$ 
  \[\delta n \geq \hat a_{\ell} = \start + (1-\tau)^knp^k\frac{\hat a_{\ell-1}^k}{k!}\text,\]
  implying that for large enough $n$
  \[ \hat a_{t-1}\leq \hat a_{\ell-1} \leq\left(k!\frac{\delta n -\start}{(1-\tau)^knp^k}\right)^{1/k} \le
  \frac{\delta^{1/k}k}{(1-\tau)p}\text.\]
  We get
  \[ 40kp\sum_{i=t_0}^{t-1}k^{t-1-i}\hat a_i\le
  40kp\sum_{i=t_0}^{t-1}k^{t-1-i}(2k)^{-(t-1-i)}\hat a_{t-1} \leq 40kp\hat
  a_{t-1}\leq 40k^2\delta^{1/k}/(1-\tau)\text. \]
  Therefore, if $\delta$ is small enough, then $\log(1+\eps_t)\le \log(1+\eps)$, and so $\eps_t\leq \eps$,
  which proves (i).

  By (i), we have $\eps_t\leq \eps<1$, and so
  \[\eta_t\cdot\min{\{1,\gamma^{-1}\}} \geq 20k\hat a_tp >(1+\eps_t)10k\hat a_t p,\] proving (iii).
\end{proof}

\subsection{End phase -- proof of Theorem~\ref{thm:synchron}}
\label{sec:endphase}

In this subsection, we will study the effect of the inhibition parameter $\tau$ on the number of active
vertices at termination.
Theorem~\ref{thm:concentration} shows in particular that the process does not stop while $a_t=o(1/p)$ (since $a_t=o(1/p)$ implies $a_{t+1}=o(n)$), and the growth of the process during that time does not depend in any significant way on the number of inhibitory vertices.
The situation changes during the very last rounds.

\begin{lemma}
  \label{lemma:percolation}
  For every $\eps>0$ there exists a $\delta=\delta(\eps,\gamma,k,\tau)>0$ such that the synchronous bootstrap percolation
  process satisfies the following, assuming $\max{\{(1+\eps)\threshold,(\log n)^{2+\eps}\}}\leq \start\leq \delta/p$
  and $p \gg n^{-1}$. Let $\ell$ denote the the largest positive integer such that $\hat a_\ell \leq \delta n$.
  \begin{enumerate}[(i)]
    \item If $\tau < 1/(1+\gamma)$ then a.a.s.\ the process almost percolates in at most $\ell +2$ rounds.
      If moreover $p \gg \log n/n$, then the process completely percolates in at most $\ell+2$ rounds.
    \item If $\tau > 1/(1+\gamma)$ and $p\gg \log n/n$, then there exists some constant $C=C(\tau,\gamma)>0$ such that
      if $\hat a_{\ell}\geq C(\log{n})/p$, then a.a.s.\ the process stops with $(1-\eps)\hat a_{\ell}\leq a^* \leq (1+\eps)\hat a_{\ell}$.
    \item If $\tau > 1/(1+\gamma)$ and $p\gg \log n/n$, then for every $\alpha>0$, there exists a constant
      $C'>0$ such that if $C'/p\le \hat a_{\ell} \le \alpha n/(1+\eps)$, then a.a.s.~the process stops
      with $a^*\leq \alpha n$.
  \end{enumerate}
\end{lemma}

Some remarks are in order. By Lemma~\ref{lemma:time-to-1/p}, we already know that 
$\ell$ is, up to an additive constant, at most $\log_k\log_{(\start/\threshold)}(pn)$. 
Then (i) shows that the number of rounds to percolation a.a.s.~takes one of only two possible (deterministic) values $\ell+1$ and $\ell+2$. If $\hat a_\ell > C (\log n)/p$, then the proof actually implies that a.a.s.~the process percolates in exactly $\ell+1$ rounds.

Lemma~\ref{lemma:percolation} spares out the border cases \emph{(a)} $\tau= 1/(1+\gamma)$, and
\emph{(b)} $\tau>  1/(1+\gamma)$ and $\hat{a}_{\ell} \leq C'/p$.
We also do not determine the size of the final
active set for the regime $\hat{a}_{\ell} \leq C(\log n)/p$. These regimes show a slightly richer, but also
more complicated behavior. Here even a harmless factor of $1+o(1)$ in the size of the
starting set can shift\footnote{We do not give a formal proof of this fact, it follows essentially from the calculations in the proof of Theorem~\ref{thm:synchron} below.} the size of the $\ell$-th set from $\Theta(1/p)$ to $\omega((\log n)/p)$, so every effect
that depends on the property $C'/p \leq \hat{a}_{\ell} \leq C(\log n)/p$ should be considered unstable. 

\begin{proof}[Proof of Lemma~\ref{lemma:percolation}]
  From Theorem~\ref{thm:concentration} we know that we can choose $\delta > 0$ such that
  \[(1-\eps)\hat a_t \leq a_t \leq (1+\eps)\hat a_t\]
  holds for all $0\leq t\leq \ell$, where $\ell$ is as in the statement of the theorem. Also, by Lemma~\ref{lemma:basic1}, we may assume that
  \[ |E_s|\in (1\pm\delta_0)(1-\tau)s\quad \text{ and }\quad |I_s|\in(1\pm\delta_0) \tau s\]
  holds for all $s\geq \start$, for some $\delta_0=\delta_0(n)=o(1)$.
  Using the definition of the sequence $(\hat a_{t})_{t\geq 0}$ and since, by definition of $\ell$, we have
  $\hat a_{\ell+1}>\delta n$,
  we can easily check that if $n$ is large enough, then we have $a_\ell\geq \delta^{1/k}/p$.
  We will prove the three statements of Lemma~\ref{lemma:percolation} separately.

  First consider (i), that is, assume that $\tau<1/(1+\gamma)$.
  In a first step we show by a case distinction that $a_{\ell+1}=\Theta(n)$.
  Let $s:=a_\ell$ and
  let $C\in \mathbb{N}$ be a large enough constant (that we define below). Assume first that $sp \geq C$. Let $\xi=\xi(\tau,\gamma)>0$ be so small that
  $(1-\xi)^2(1-\tau) \ge (1+\xi)^2\tau\gamma$; such a choice is possible since $\tau<1/(1+\gamma)$.
  Then the assumption that
  \[ |E_s|\in (1\pm\delta_0)(1-\tau)s\quad \text{ and }\quad |I_s|\in(1\pm\delta_0) \tau s\]
  implies in particular that $(1-\xi)|E_s|p \geq (1+\xi)|I_s|\gamma p$.
  By the Chernoff bounds (Lemma~\ref{lemma:Chernoff}), we have
  \[ \Pr[N^-_s(v)> (1+ \xi/2)|I_s|\gamma p] \leq e^{-\frac1{12}\xi^2|I_s|\gamma p} \leq e^{-\frac1{12}\xi^2(1-\delta_0)\tau \gamma sp} \]
  and 
  \[ \Pr[N^+_s(v)< (1- \xi)|E_s|p] \leq e^{-\frac13\xi^2|E_s|p} \leq e^{-\frac13\xi^2(1-\delta_0)(1-\tau) sp}\text. \]
  Since $sp\geq C$, by choosing $C$ large enough, we can assume that
  both these probabilities are at most $1/3$. Then with probability at least
  $1/3$, we have
  \[ N^+_s(v) \geq (1-\xi)|E_s|p
  \geq (1+\xi)|I_s|\gamma p
  \geq k +(1+\xi/2)|I_s|\gamma p
  \geq k+N^-_{s}(v)\text,
  \]
  where we used that for large enough $C$, we have $\xi|I_s|\gamma p/2\geq  k$. For the second case assume now that
  $\delta^{1/k}\leq sp < C$, which implies in particular that $p<C/s=o(1)$.
  In this case, the probability that $N^-_s(v) = 0$ is at least
  \[(1-\gamma p)^{|I_s|} \geq (1-\gamma p)^{s}\geq (1-\gamma p)^{C/p}\geq 2^{-2C\gamma },\]
  and the probability that $N^+_s(v)\geq k$ is at least
  \[ \binom{|E_s|}{k}p^k (1-p)^k\geq \frac{|E_s|^k}{2k^k}p^k \geq \frac{p^ks^k(1-\tau)^k}{4k^k}\geq \frac{\delta(1-\tau)^k}{4k^k}\text.\]
  So both probabilities are bounded from below by positive constants. Since $N^+_s(v)$ and
  $N^-_s(v)$ are conditionally independent, we see that with positive probability, we have
  $N^+_s(v)\geq k+N^-_s(v)$. Summarizing, we proved in both cases that for every vertex $v$, $N^+_s(v)\geq k+N^-_s(v)$ occurs independently with some nonzero constant probability.
  Another application of Chernoff thus implies that after round $\ell+1$, a.a.s.~a linear fraction of all
  vertices is active. 
  Since $p = \omega(n^{-1})$, we then have $\Pr[N^+_{a_{\ell+1}}(v) \geq k+N^-_{a_{\ell+1}}(v)] = 1-o(1)$
  for all vertices $v\in [n]$. This implies that in round $\ell+2$, there are $n-o(n)$ active vertices.
  If we assume additionally that $p = \omega(\log n/n)$, then using the
  Chernoff bounds and the union bound, we actually obtain that a.a.s.\ we have
  $N^{+}_{a_{\ell+1}}(v)\geq k+N^{-}_{a_{\ell+1}}(v)$
  for all $v\in [n]$, which proves that all vertices are active in round $\ell+2$, showing $(i)$.

  To show (ii), assume $\tau>1/(1+\gamma)$ and that $\hat{a}_{\ell} \geq C(\log n)/p$ holds for some large constant $C=C(\tau,\gamma)>0$ (chosen below), so that
  $s:=a_\ell\geq (1-\eps)C(\log n)/p$.
  To prove that the process stops with $s$ active vertices, it is enough to show that every vertex
  $v\in [n]$ is such that $N^-_s(v)\geq N^+_s(v)$. Fix any vertex
  $v\in [n]$ and choose a constant $\xi=\xi(\tau,\gamma)>0$ so small that
  \[ (1+\xi)|E_s|p \leq (1+\xi)(1+\delta_0)(1-\tau)sp \leq (1-\xi)(1-\delta_0)\tau\gamma s p\leq (1-\xi)|I_s|\gamma p\text;\]
  such a choice is possible since $\tau>1/(1+\gamma)$. If $C$ is sufficiently large, then by the Chernoff bounds
  (Lemma~\ref{lemma:Chernoff}) we get
  \[\Pr[N^-_s(v)< (1-\xi)|I_s|\gamma p]\leq e^{-\xi^2(1-\delta_0)\tau \gamma sp/3}\ll n^{-1}\text,\]
  and
  \[\Pr[N^+_s(v)> (1+\xi)|E_s|p] \leq e^{-\xi^2(1-\delta_0)\tau sp/3}\ll n^{-1}\text.\]
  So by the union bound, a.a.s.~every vertex $v\in [n]$ satisfies $N^-_s(v)\geq N^+_s(v)$ and the process will stop.

  Finally, for (iii), let $\tau >1/(1+\gamma)$ and suppose that we are given some $\alpha >0$.
  Let $C'=C'(\alpha)$ be large enough and assume that
  $$
  (1-\eps)C'/p \;\le\; (1-\eps)\hat{a}_{\ell}\;\le\;  a_\ell\;\leq\; (1+\eps)\hat a_{\ell}
  \;\leq\; \alpha n.
  $$
  If $(1+\eps)\hat a_\ell\geq \alpha n/2$, then by (ii), the process will stop with
  $a^*=a_\ell$ active vertices, so assume from now on that $a_\ell\leq (1+\eps)\hat a_\ell<\alpha n/2$.

  First, if there is no $t\geq \ell$ such that
  $a_t \geq \alpha n/2$, then we are done. Otherwise, let $t_0$ be the smallest $t\geq \ell$ with
  this property. The same arguments as in (ii) show that the process stops with $a^* = a_{t_0}$.
  Thus, it suffices to show that $a_{t_0}\leq \alpha n$.
  To prove this, it is enough to show that with probability tending to one, 
  we have $a_{t+1}\le a_{t}+ \alpha n/2$ for all $t\geq \ell$. To see this recall that
  $a_t \geq a_{\ell}\geq (1-\eps)C'/p$. Thus, if we choose $C'$ large enough, then we have
  $\Pr[N^+_s(v) \geq N^-_s(v)]\leq \alpha /4$ for every vertex $v\in [n]$ and for every $s\geq a_t$, using
  $\tau>1/(1+\gamma)$.
  Then, by the Chernoff bound, the probability that $a_{t+1}-a_t\geq \alpha n/2$ is
  $o(n^{-1})$. Since there can be at most $n$ rounds in total until the process stops (there are only $n$ vertices), the union bound easily shows
  that a.a.s., $a_{t+1}\le a_t+ \alpha n/2$ holds for all $t\geq \ell$, completing the proof.\qedhere

\end{proof}

\begin{proof}[Proof of Theorem~\ref{thm:synchron}]
  Observe that Corollary~\ref{cor:babysteps}, Lemma~\ref{lemma:percolation} allows us to restrict ourselves to the case $\tau > 1/(1+\gamma)$.
  Given any real number $\hat a_0$, we can define a sequence $(\hat a_t)_{t\geq 0}$ by \eqref{eq:defahat}, as
  in the statement of Theorem~\ref{thm:concentration}.
  Our first goal is to show that this sequence is sufficiently robust against rounding down the starting value $\hat a_0$.

  For this, fix any $C_2>C_1>0$, and assume that $C_1\threshold \leq \hat a_0\leq C_2\threshold$ is any real number. Denote by $\ell$ the largest positive integer
  $t$ for which $\hat{a}_t \le  n$. From Lemma~\ref{lemma:time-to-1/p} we
  know that \[\ell = \log_k\log_{(\hat a_0 /\threshold)}(pn) +O(1)=\log_k\log(pn) +O(1)\text.\]
  Let $(\hat b_t)_{t\geq 0}$ denote the sequence defined by the same recursion as
  $\hat a_t$, but with an initial value of $\lfloor\hat a_0\rfloor$, i.e., $\hat b_0 = \lfloor \hat a_0 \rfloor$ and
  $\hat b_{t+1} = \hat b_0+ (1-\tau)^k np^k \hat b_t^k/k!$.
  We will show by induction that for all $t\geq 0$, we have
  $\hat b_t / \hat a_t \geq (1-1/\hat a_0)^{k^t}$. For $t=0$ this immediately follows from $\hat b_0 \geq \hat a_0-1$.
  For the inductive step assume $\hat b_{t-1}/\hat a_{t-1}\geq (1-1/\hat a_0)^{k^{t-1}}$.
  Using Equation~\eqref{eq:growth} on page~\pageref{eq:growth}, we have
  \[\frac{\hat{b}_t}{\hat{a}_t} = \frac{\hat b_0+(\hat{b}_{t-1}/\thebox)^k\thebox/k}{\hat{a}_0+(\hat{a}_{t-1}/\thebox)^k\thebox/k}
  \geq \frac{\left(1-\frac{1}{\hat a_0}\right)\hat a_0+\left(1-\frac{1}{\hat a_0}\right)^{k^t}(\hat{a}_{t-1}/\thebox)^k\thebox/k}{\hat{a}_0+(\hat{a}_{t-1}/\thebox)^k\thebox/k}\geq \left(1-\frac{1}{\hat a_0}\right)^{k^t}\text,\]
  as claimed.
  Thus the error  in $\hat a_{\ell}$ caused by
  rounding $\hat a_0$ down to the next integer satisfies
  \[ 1\geq \frac{\hat b_{\ell}}{\hat a_{\ell}}
  \geq\left(1-\frac{1}{\hat a_0}\right)^{k^{\ell}}\geq
  \left(1-\frac{1}{C_1(\log{n})^{2+\eps}}\right)^{\Theta(\log(pn))}\to 1\text,\]
  by the assumption that $\hat a_0 \geq C_1\threshold\geq C_1(\log n)^{2+\eps}$.
  This means that for the asymptotic size of $\hat a_{\ell}$, it does not matter
  whether $\hat a_0$ is rounded down to the next smallest integer or not.

  To complete the proof of Theorem~\ref{thm:synchron}, we will show that for every
  constant $C_1$, there exists a constant $C_2$ such that for every function $\log n/p\ll f(n)\ll n$,
  there exists a function $C_1\leq c(n)\leq C_2$ such that a.a.s., the process with $\start = \lfloor c(n)\Lambda\rfloor$ 
  stops with $(1+o(1))f(n)$ active vertices. Observe that it suffices to consider constants  
  $C_1$ that are sufficiently large so that the inequalities below hold.

  Consider the process with $\start = C_1\Lambda$. Recall that we assume that $\threshold\geq (\log n)^{2+\eps}$ holds for some constant $\eps>0$. Since we may assume that $C_1\geq 1+\eps$, Theorem~\ref{thm:concentration}
  implies that there exists some $\delta>0$ such that a.a.s.,
  \[ (1-\eps)\hat a_t \leq a_t\leq (1+\eps)\hat a_t \]
  holds for all $0\leq t \leq \ell$, where 
  $\ell$ is the largest integer such that $\hat a_{\ell}\leq \delta n$.
  Define $\ell_0$ to be the largest integer such that
  $\hat a_{\ell_0}\leq f(n)/(1+\eps)$, and note that, since $f(n)\ll n$, we have $\ell_0\leq \ell$ for all large enough $n$.
  Thus we have $a_{\ell_0}\leq (1+\eps)\hat a_{\ell_0}\leq f(n)$ a.a.s..

  Observe also that for large enough $n$, we have
  $f(n) \leq \hat a_{\ell_0} np$, which is obvious if $\hat a_{\ell_0} \geq 1/p$ and
  otherwise follows from
  \[ f(n)/(1+\eps)\leq \hat a_{\ell_0+1}=\hat a_0+(1-\tau)^k\hat a_{\ell_0}^k\frac{np^k}{k!}\leq \hat a_0 + np\hat a_{\ell_0}/k!
  \leq np\hat a_{\ell_0}/(1+\eps). \]

  We will show that if one multiplies $\hat a_0$ with a large enough constant factor $c_0$, then $\hat a_{\ell_0}$
  increases by a factor of $\omega(pn)$. This will imply, by the intermediate value theorem, that there exists
  some $c=c(n)\in [C_1,c_0C_1]$ such that a starting value $\hat a_0 = c\thebox$ results in $\hat a_{\ell_0} = f(n)$. Then, by
  the argument above, and by Lemma~\ref{lemma:percolation} (ii) (using $f(n)\gg (\log n)/p$), the process with $\start = \lfloor c\thebox\rfloor$
  will stop after $\ell_0$ rounds with $(1+o(1))f(n)$ active vertices. Since $C_1$ is an arbitrary constant and since $\thebox=\Theta(\threshold)$, this will complete the proof of the theorem.

  So consider a sequence $(\hat b_t)_{t\geq 0}$ defined by $\hat b_0 = c_0C_1\thebox$ and by the same recursion \eqref{eq:defahat}, with $\hat a$ replaced by $\hat b$. Our goal is to show that $\hat b_{\ell_0}/\hat a_{\ell_0} = \omega(pn)$.
  Write $\hat b_t = c_t \hat a_t$  and  $\hat{a}_t = x_t\thebox$. Using \eqref{eq:growth} and the fact that $x_t$ is monotonically increasing, we see that for all $t\geq 0$, we have
  \[
  c_{t+1} = \frac{\hat{b}_{t+1}}{\hat{a}_{t+1}} =
  \frac{c_0C_1+(c_tx_t)^k/k}{C_1+x_t^k/k}
  \geq c_{t}^k\left(1-\frac{C_1}{C_1 +x_t^k/k}\right)\geq c_t^k\left(1-\frac{C_1}{C_1+C_1^k/k}\right).
  \]
  In particular, if $C_1$ and $c_0$ are large enough, then we have $c_1\geq 2c_0$ and $c_t\geq c_{t-1}^k/2$ for all $t>0$.
  By induction it follows that for all $t>0$, we have
  \[ c_t \geq \frac{c_{1}^{k^{t-1}}}{2^{(k^{t-1}-1)/(k-1)}} \geq c_0^{k^{t-1}}.\]
  Since $f(n)\geq 1/p$, Lemma~\ref{lemma:time-to-1/p} tells us that $\ell_0\geq \log_k \log (pn)-O(1)$, where the constant in the $O(1)$ term does not depend on $c_0$. Therefore,
  if $c_0$ is large enough, we get
  \[ \frac{\hat b_{\ell_0}}{\hat a_{\ell_0}}= c_\ell \geq c_0^{\Omega(\log {pn})}=\omega(pn), \]
  completing the proof.
\end{proof}

\section{Asynchronous Bootstrap Percolation}\label{sec:async}

In the second part of the paper we consider the bootstrap percolation process with an additional temporal component. More precisely, we assume that all edges have independent delays distributed according to $\Exp(1)$.
Recall that these transmission delays correspond to the random variables $\Phi_{iv}$ in the probability space introduced in Section~\ref{sec:probspace}. 

The main difference of this model to the synchronous case studied in the previous section is that the activation no longer takes place in rounds, but that vertices turn active at individual times. 
Recall that we write $t_s$ for the time at which the $s$-th vertex turns active. Note that
we may assume without loss of generality that no two vertices become active at the same time
(except for the vertices in the starting set).


\subsection{Initial phases}

The goal of this subsection is to describe the behavior of the process in the range where few vertices
are active. In this range, inhibition does not play an important role. 

\begin{lemma}\label{lemma:asyncexplosion}
  For every $\eps>0$, there exists a constant $T=T(\eps,k)>0$ such that
  the asynchronous process with $\start\geq (1+\eps)\threshold$ and $n^{-1}\ll p \ll n^{-1/k}$ satisfies the following. For every constant $C>0$, a.a.s.,
  \begin{enumerate}[(i)]
    \item $t_{C/p}\leq T$,
    \item $t_{C/p}-t_{1000\Lambda}\leq 1+o(1)$, and
    \item $t_{C/p}-t_{s}=o(1)$ for all $s=\omega(\threshold)$.
  \end{enumerate}
\end{lemma} 
\begin{proof}
  Let $\delta_0\in (0,1/2)$ be a sufficiently small constant.
  By Lemma~\ref{lemma:basic1}, and since $p\ll n^{-1/k}$ implies $\start\ge \threshold = \omega(1)$, we can condition the process on the event $\mathcal E$ that
  \[ |E_s|\in (1\pm\delta_0)(1-\tau)s \]
  holds for all $s\geq \start$. Moreover, by Lemma~\ref{lemma:lower} (with $\eta=1/2$), we know that for every $\start \leq
  s=x\Lambda\leq 1/(20kp)$ and $r>0$, we have
  \begin{equation}\label{eq:lemlower}\Pr[L_s(r) <\Pr[\Phi\leq r]^kx^{k}\thebox/(2k)\mid \mathcal E]< e^{-c\Pr[\Phi\leq r]^kx^k\thebox}\text,
  \end{equation}
  for some positive constant $c=c(\tau,k)$ that is independent of $r$ and $x$.

  From Corollary~\ref{cor:babysteps}, we know that a.a.s.\ the process reaches at least $s_0= \min{\{1000\thebox,\delta/p\}}$ active vertices after some constant time $T_0=T_0(\eps,k)$, for an appropriate $\delta=\delta(\eps,\gamma,k)>0$. 
  Starting from $T_0$, we examine successive time intervals of lengths $1/2, 1/4,1/8 \ldots$, respectively, and compute the number of active excitatory vertices after each interval. 
  So, for every $i\geq 1$, 
  define $T_i = T_0+\sum_{j=1}^{i}2^{-j}$ and let $s_i$ be the number of vertices active at time $T_i$.
  We claim that a.a.s.~we have
  \begin{equation}\label{eq:si}s_i\geq \min{\{10^{i+3}\thebox,\delta/p\}}
  \end{equation} for every $i\geq 0$. Write $\mathcal S_i$ for the event that \eqref{eq:si} holds for $i$.
  Since the lengths of our time intervals sum up to $1$, the occurrence of $\bigcap_{i\geq 0}
  \mathcal S_i$ implies that at time $T_0+1$ there are at least $\delta/p$ active vertices.
  We will show  by induction that 
  \begin{equation}\label{eq:progress}
    \Pr[\mathcal{S}_i\mid\mathcal E]\ge 1-\sum_{j=1}^i e^{-c2^j\thebox}
  \end{equation}
  holds for all $i\geq 0$.
  The case $i=0$ is evident by the choice of $s_0$, so let us assume that \eqref{eq:progress}
  holds for some $i\ge 0$. Let $\Delta_i:= T_{i+1}-T_i=2^{-(i+1)}$.

  Write $\mathcal A$ for the event that $s_i\geq \delta/p$.
  Since $s_i\geq \delta/p$ implies $s_{i+1}\geq \delta/p$, we have
  \[\Pr[\mathcal S_{i+1}\mid \mathcal S_i\cap \mathcal E\cap \mathcal A] = 1.\]

  On the other hand, if $\mathcal A$ does not occur, then $\mathcal S_i$ implies that
  $s_i\geq 10^{i+3}\thebox$, and in this case, we have
  $s_{i+1}\geq \min{\{\start+L_{10^{i+3}\thebox}(\Delta_i),10^{i+4}\thebox\}}$, by the definition of
  $L_{10^{i+3}\thebox}(\Delta_i)$.
  By \eqref{eq:lemlower}, we have
  \[ \Pr[L_{10^{i+3}\thebox}(\Delta_i) < \Pr[\Phi\leq \Delta_i]^k10^{k(i+3)}\thebox/(2k)\mid \mathcal E] <e^{-c\Pr[\Phi\leq  \Delta_i]^k10^{k(i+3)}\thebox}\]
  for a suitable constant $c>0$.  Using the fact that $\Phi\sim\Exp(1)$, that $k\geq 2$, and that $e^{-x}\le 1-x/2$ for $0<x\le 1/2$ we get
  \[ \Pr[\Phi\leq  \Delta_i]^k10^{k(i+3)}
  = \left(1-e^{-2^{-(i+1)}}\right)^k10^{k(i+3)}\geq 2^{-k(i+2)}\cdot 10^{k(i+3)}\ge 2k10^{i+4} \]
  for all $i\geq 0$.
  Then we get
  \[ \Pr[\mathcal S_{i+1} \mid \mathcal S_i\cap \mathcal E\cap \overline{\mathcal A}]\geq 1-e^{-c2k10^{i+4}\thebox}\geq 1-e^{-c2^{i+1}\thebox}.\]
  Therefore,
  \[\Pr[\mathcal S_{i+1}\mid \mathcal E] \geq \Pr[\mathcal S_i\mid \mathcal E]
  \Big(\Pr[\mathcal A]+\Pr[\overline{\mathcal A}] (1-e^{-c2^{i+1}\thebox})\Big)\geq 1-\sum_{j=1}^{i+1}e^{-c2^{j}\thebox}\text,\]
  using the induction hypothesis. This completes the proof of \eqref{eq:progress}.

  Finally, since
  $\sum_{j=1}^{\infty}e^{-c2^j\Lambda} = o(1)$, it follows by the union bound that
  $\bigcap_{i\geq 0}\mathcal S_i$ occurs with probability $1-o(1)$, i.e., that a.a.s.~at least $\delta/p$ vertices
  are active at time $T_0+1$. Since $T_0$ is the first time at which $\min\{1000\Lambda,\delta/p\}$ vertices
  are active, this also shows that the time to go
  from $1000\thebox$ to $\delta/p$ active vertices is at most $1$. Additionally, the same proof shows that a.a.s.~for every $i\geq 0$, the time from $s_i=10^{i+2}\Lambda$ to $\delta/p$ active vertices is at most $T_0+1-T_i$, which implies that
  a.a.s., the time to go from $\omega(\threshold)$ to $\delta/p$ active vertices is $o(1)$.

  Let $T'$ denote the earliest time at which $\delta/p$ vertices are active. To complete the proof of (i--iii), it suffices to show that for all constants $\rho>0$ and $C>0$, a.a.s.\ at least $C/p$ vertices
  are active at time $T'+\rho$. Note that for a fixed vertex $v\in[n]$, the events
  $N^-_{C/p} = 0$ and $N^+_{\delta/p}=k$ and $\max_{i\in E_{\delta/p}}\{X_{iv}\Phi_{iv}\leq \rho\}$
  occur simultaneously with some positive constant probability.
  Thus, by the Chernoff bounds (Lemma~\ref{lemma:Chernoff}), we get that a.a.s.~a constant fraction of all vertices, say $c n$ vertices for some $c=c(k,\gamma,\rho,C)$ satisfies these three conditions. Hence, by time $T'+\rho$, at least $\min\{cn,C/p\}$ vertices are active. This proves the claim since $C/p\leq cn$ for large enough $n$.
\end{proof}


\subsection{End phase -- proof of Theorem~\ref{thm:async-percolation}}

Write $a(t)$ for the number of vertices that are active at time $t$ and
write
\[S^+(v,t) := |\{0\leq i \leq a(t) \mid X^+_{iv}=1\text{ and }\Psi_i=+1\text{ and }t_i+\Phi_{iv}\leq t\}|\]
and
\[S^-(v,t) := |\{0\leq i \leq a(t) \mid X^-_{iv}=1\text{ and }\Psi_i=-1\text{ and }t_i+\Phi^-_{iv}\leq t\}|\]
for the number of excitatory/inhibitory signals that have reached vertex $v$ at time $t$. For brevity, we will also write $S(v,t):= S^+(v,t)+S^-(v,t)$.

Observe that a vertex is part of the final active set exactly if there is some time $t\geq 0$ such that
$S^+(v,t)-S^-(v,t)\geq k$. Thus, our main goal will be to describe the evolution of
the random variables $S^+(v,t)-S^-(v,t)$, for different vertices $v$.

We start by proving some properties that are satisfied by most vertices during the end phase of the process.

\begin{lemma}\label{lemma:endgame}
  For every $\eps>0$, the asynchronous process with $\start\geq (1+\eps)\threshold$
  and $n^{-1}\ll p\ll n^{-1/k}$ satisfies the following. For all constants $\xi\in (0,1/3)$
  and $C_0>0$, and for every sufficiently large constant $C>0$, if $n\in\mathbb N$ is large enough,
  then with probability at least 
  $1-\xi$, all but at most $\xi n$ vertices $v\in [n]$ satisfy:
  \begin{enumerate}[(i)]
    \item $N^+_s(v)\in (1\pm \xi)(1-\tau) sp$ \quad and \quad $N^-_s(v)\in (1\pm \xi)\tau \gamma sp$ \quad for all $s\geq C/p$,
    \item $S(v,t_{C/p}) = 0$, and
    \item $S(v,t_{1/p}+\xi)\geq  C_0$.
  \end{enumerate}
\end{lemma}
\begin{proof}
  Fix some vertex $v\in [n]$. We will show that for $C$ chosen sufficiently large, the vertex $v$
  satisfies (i)--(iii) with probability at least $1-\xi^2$. 
  Then it will follow from Markov's
  inequality that the probability that more than $\xi n$ vertices fail to satisfy (i)--(iii) is at most
  $\xi$.

  We first show (i).
  Recall that by Lemma~\ref{lemma:basic1} we may assume that with probability $1-\xi^2/8$ we have 
  \begin{equation}\label{eq:endgame:e1}
    \abs{E_s}\in (1\pm \xi/8)(1-\tau)s\quad\text{for all $s\ge (1+\eps)\threshold$,}
  \end{equation}
  where, as before, $E_s$ denotes the set of excitatory vertices among the first $s$ active vertices. We know that
  $N^+_s(v)$ is binomially distributed with parameters $|E_s|$ and $p$. That is, given \eqref{eq:endgame:e1} an application
  of Chernoff bounds (Lemma~\ref{lemma:Chernoff}) gives us
  $$
  \Pr[ N^+_s(v)\not\in (1\pm \xi/2)(1-\tau) sp] \le e^{-c\xi^2 (1-\tau)sp},
  $$
  for some constant $c$ that does not depend on $\xi$. 
  By the union bound we deduce that the claim holds for all $s=(C+i)/p$ with $i\in\mathbb N_0$, as 
  $\sum_{i\ge 0}  e^{-c\xi^2 (1-\tau)(C+i)} = O(e^{-c\xi^2 (1-\tau)C})$ and the term on the right hand side can be made smaller than, say, $\xi^2/8$ by choosing $C=C(\xi,\tau)$ sufficiently large.
  Now consider an $s=(C+i+r)/p$ for some $r\in(0,1)$ and let $\bar s = (C+i+1)/p$. Then
  $$
  N^+_s(v)\le N^+_{\bar s}(v) \le (1+ \xi/2)(1-\tau) \bar sp \le (1+ \xi)(1-\tau)  sp, 
  $$
  whenever $C=C(\xi,\tau)$ is sufficiently large. The lower bound for $N^+_s(v)$ follows similarly. The statement for $N^-_s(v)$ follows similarly, with $|E_s|$ replaced by $|I_s|$ and
  $p$ replaced by $\gamma p$. This then shows that
  (i) holds for $v$ with probability at least $1-\xi^2/4$.

%
%
%
  Next we show (ii). The statement is trivial if $\start \ge C/p$, so assume otherwise. Assume further for the time being that $N_{C/p}(v)\leq (1+\xi)(1-\tau+\tau\gamma)C$, i.e., that $v$ has only constantly
  many active neighbors at time $t_{C/p}$. By (i),  $v$ satisfies this condition
  with probability at least $1-\xi^2/4$.

  Choose $\delta = \delta(\xi)>0$ so that $N_{\delta/p}(v)=0$
  holds with probability at least $1-\xi^2/5$.
  Since $\delta/p=\omega(\threshold)$ (this is implied by $p=\omega(n^{-1})$),
  Lemma~\ref{lemma:asyncexplosion} tells us that $t_{C/p}-t_{\delta/p}=o(1)$ a.a.s.. 
  The probability that one of the (constantly many) active neighbors of $v$ at time $t_{C/p}$
  has already sent its signal to $v$ is thus $o(1)$.   Therefore, the probability that $v$ satisfies both (i) and (ii) is at least $1-\xi^2/2$ for sufficiently large $n$.

  Lastly, let us prove (iii). Assume that $v$ satisfies both (i) and (ii). By the previous point, this happens with probability at least
  $1-\xi^2/2$.
  In particular, assume that $N^+_{C/p}(v)\geq (1-\xi)(1-\tau)C$.
  If $C$ is large enough, then with probability at least $1-\xi^2/4$,
  at least $C_0$ excitatory neighbors of $v$ will send their signal to $v$ before time $t_{C/p}+\xi/2$,
  i.e., $S(v,t_{C/p}+\xi/2)\geq C_0$. However, by Lemma~\ref{lemma:asyncexplosion}, we have
  $t_{C/p}+\xi/2 \leq t_{1/p}+\xi$, which
  shows that $v$ satisfies (i)--(iii) with probability at least $1-\xi^2$.
  As noted above, an application of Markov's inequality completes the proof.
\end{proof}

Before we come to the technical part of the proof, we give an intuition for the result. Let us consider a typical vertex $v$ at time $t_{C/p}$. The previous lemma shows that,
although $v$ has many active neighbors at time $t_{C/p}$, none of their signals has arrived at vertex $v$ at that time.
Moreover, we can assume that throughout the process, roughly the correct fraction of the neighbors of $v$ are excitatory.
I.e., when $s$ vertices are active there are about $(1-\tau) sp$ excitatory neighbors and $\tau\gamma sp$ inhibitory ones. 
Recall that we assumed that the delays
(the variables $\Phi_{iv}$) are exponentially distributed, that is, they are \emph{memoryless},
which  means that every neighbor whose signal has not yet arrived
is \emph{equally likely} to be the next to deliver its signal. Therefore, we would expect that, as the signals come in,
the difference $S^+(v,t)-S^-(v,t)$ performs a random walk with a bias close to
$\frac{1-\tau}{1-\tau+\tau\gamma}$, in which case
the probability that
the vertex $v$ becomes active is just the probability that such a random walk ever reaches $k$. This
is the idea for the rest of the analysis. We will need the following basic facts about random walks.
\begin{lemma}\label{lemma:random-walk-app}
  Let $X_1,X_2,\dotsc,X_n$ be a sequence of independent random variables, each of which
  is equal to $1$ with probability $\beta \in [0,1]$ and $-1$ otherwise. Consider the
  \emph{biased random walk} $Z_i = X_1+X_2+\dotsb+X_i$. Then there exists for every $\eps > 0$ and $k\in\N$ a constant $C_0 = C_0(\eps,k)$ such that the following is true:
  $$\Pr[\forall n\geq C_0 : Z_n \in (2\beta-1)n \pm \eps n]\ge 1-\eps$$ and
  $$\Pr[\exists i \le C_0 \text{ s.t. } Z_i=k] \in (1\pm\eps) \cdot \min\left\{1,\frac{\beta^k}{(1-\beta)^k}\right\}.$$
\end{lemma}
\begin{proof}
  The first fact follows immediately from Chernoff bounds and implies the second fact whenever
  $\beta > 1/2$. For $\beta\le1/2$ the second fact follows from \cite{Grimmett2001}, Problem 5.3.1.
\end{proof}

For every vertex $v\in [n]$ and every $i\in\mathbb N$, define
$X_i^{(v)}$ to be $1$ if the $i$-th signal arriving in $v$ is excitatory, and $-1$ otherwise. Here
we assume that in the asynchronous process,
no two signals arrive simultaneously, which is the case with probability $1$.
Then we can define $Z_i^{(v)}:= X_1^{(v)}+X_2^{(v)}+\dotsb+X_i^{(v)}$, and we know that
the vertex $v$ becomes active
with the arrival of the first signal that causes $Z_i^{(v)}$ to become $k$, if such a signal exists. As outlined before, we will show that $Z_i^{(v)}$ follows (essentially) a random walk with bias 
\begin{equation}
  \label{eq:beta}
  \beta := \frac{1-\tau}{1-\tau+\gamma\tau}.
\end{equation}
If $\tau \geq 1/(1+\gamma)$, then $\beta \leq 1/2$, and by Lemma~\ref{lemma:random-walk-app} we would expect that roughly $n\beta^k/(1-\beta)^k $ vertices are activated. There are two problems which complicate the analysis: the first being that the processes
$(Z_i^{(v)})_{i\in \mathbb N}$ and $(Z_i^{(u)})_{i\in \mathbb N}$ are not independent for different
vertices $u$ and $v$, and the second being that for a fixed vertex $v$, the variables $X_i^{(v)}$ and $X_j^{(v)}$ are not independent for $i\neq j$, meaning that $(Z_i^{(v)})_{i\in \mathbb N}$ is not a true random walk.
However, the following lemma tells us that at least for the first $C_0$ incoming signals these problems do not matter.

\begin{lemma}\label{lemma:walks}
  For every $0<\eps,\zeta < 1/3$, there exists some $T=T(\eps,k)>0$ independent of $\zeta$
  such that for every large enough constant $C_0>0$, the asynchronous percolation process with
  $\start \geq (1+\eps)\threshold$
  and $n^{-1}\ll p\ll n^{-1/k}$ satisfies the following: for every large enough $n\in\mathbb N$,
  with probability at least $1-\zeta$,
  \begin{itemize}
    \item if $\tau > 1/(1+\gamma)$, then $a(T) \geq \beta^kn/(1-\beta)^k - \zeta n$ and at most $\beta^kn/(1-\beta)^k + \zeta n$ vertices get activated by their first $C_0$ incoming signals.
      Moreover, all but at most $\zeta n$ vertices $v$ are such that
      $Z_{C_0}^{(v)}\leq -(1-\zeta)(1-2\beta)C_0$.
    \item if $\tau \leq 1/(1+\gamma)$, then $a(T)\geq n-\zeta n$.
  \end{itemize}
\end{lemma}

\begin{proof}
  Assume first that $\tau >1/(1+\gamma)$ and note that this is
  equivalent to $0< \beta < 1/2$, where $\beta$ is as in \eqref{eq:beta}.
  Let $\xi=\xi(\zeta) > 0$ be sufficiently small (to be  fixed below) and
  choose $\eta=\eta(\zeta)>0$ such that $0<\beta-\eta\leq \beta + \eta \leq 1/2$ and
  \begin{equation}\label{eq:async:walk1}
    \left(\frac{\beta}{1-\beta}\right)^k - \frac{\zeta}{2} \le \left(\frac{\beta-\eta}{1-(\beta-\eta)}\right)^k-\eta \le \left(\frac{\beta+\eta}{1-(\beta+\eta)}\right)^k + \eta \le \left(\frac{\beta}{1-\beta}\right)^k + \frac{\zeta}{2}.
  \end{equation}
  Moreover, we may assume that $\eta < \zeta(1-2\beta)/3$, which implies $(2(\beta+\eta)-1)C_0+\eta C_0 \le (1-\zeta)(2\beta-1)C_0 $ (note that $2\beta-1$ is negative). We will apply  Lemma~\ref{lemma:random-walk-app} for $\beta-\eta$ and for $\beta+\eta$.
  Clearly, whenever 
  $C_0=C_0(\eta)=C_0(\zeta)$ is sufficiently large, then for both values the probability in Lemma~\ref{lemma:random-walk-app} is 
  within $\eta$ of the limit if we only consider the first $C_0$ terms. 
  We will also apply Lemma~\ref{lemma:random-walk-app} with $\beta+\eta$, and we can assume
  that $C_0$ is so large that the probability that $Z_{C_0}\in (2\beta-1)C_0 \pm \eta C_0$ is
  at least $1-\xi$. 

  Let $C=C(\xi,C_0) = C(\zeta)$ be so large that $C_0 \le \xi (1-\xi)(1-\tau) C$ and that
  Lemma~\ref{lemma:endgame}  guarantees that with probability $1-\xi$ there exists 
  a set $V_0\subseteq [n]$ of size at least $|V_0| \ge (1-\xi)n$ such that for all $v\in V_0$ we have
  \begin{enumerate}[(i)]
    \item $N^+_s(v)\in (1\pm \xi)(1-\tau) sp$ \quad and \quad $N^-_s(v)\in (1\pm \xi)\tau \gamma sp$ \quad for all $s\geq C/p$,
    \item $S(v,t_{C/p}) = 0$, and
    \item $S(v,t_{1/p}+\xi)\geq  C_0$.
  \end{enumerate}
  We will prove that
  for every vertex $v\in V_0$ and every time $t\geq 0$ such that $S(v,t)< C_0$, the first
  signal arriving in $v$ after time $t$ is excitatory with probability within $\beta\pm \eta$.
  Moreover, we will show:  $(\star)$ these bounds  hold regardless of the states of all other vertices. 

  Before proving this claim we show that this suffices to prove the first bullet point.
  Applying 
  Lemma~\ref{lemma:random-walk-app} with respect to $\beta-\eta$ and $\beta + \eta$, together with our assumptions on $C_0$  and Lemma~\ref{lemma:endgame} we observe
  that the probability that a vertex $v\in V_0$ gets active by receiving the first $C_0$ signals is within $(\beta\pm \eta)^k/(1-\beta\pm\eta)^k \pm \eta= \beta^k/(1-\beta)^k\pm \zeta/2$ by~\eqref{eq:async:walk1}. 
  By applying Chernoff bounds (which we may, because of $(\star)$), and since we may choose $\xi = \xi(\zeta)$ small enough, this then implies 
  that a.a.s.\ at least $(1- \xi)(\beta^k/(1-\beta)^k -\zeta/2 )|V_0| \geq \beta^k/(1-\beta)^kn- \zeta n$ vertices become
  active with one of the first $C_0$ incoming signals. (Note that the error  that we get from Chernoff is in the order
  $e^{-\Theta(n)}$, which is smaller than $\xi$ for all large enough $n$.)
  Similarly, a.a.s.\ at most $(1+\xi)(\beta^k/(1-\beta)^k+\zeta/2)|V_0| + (n-|V_0|)
  \leq \beta^k/(1-\beta)^k+\zeta n/2$ vertices become active by one of the first $C_0$ signals.
  In addition, Lemma~\ref{lemma:random-walk-app} (with $\beta +\eta$) and  Chernoff bounds show that for at least
  $(1-\xi)|V_0|\geq n-\zeta n$ vertices, we have $Z_{C_0}^{(v)}\leq (1-\zeta)(2(\beta+\eta)-1)C_0 +\eta C_0 \leq (1-\zeta)(2\beta-1)C_0$.
  Since $t_{1/p}+\xi < t_{1/p}+1$ can be upper bounded by a constant $T$ by Lemma~\ref{lemma:asyncexplosion}, this implies the claim of
  the first bullet point.

  So consider some $v\in V_0$ and $t\geq 0$ with $S(v,t)< C_0$.
  Assume that the first signal that arrives at $v$ after time $t$
  does so at time $t^*$.
  Let $s\in [n]$ be such that $t^*\in (t_s,t_{s+1}]$. By (ii), we know that we have
  $s+1>  C/p$, as no signals arrive before time $t_{C/p}$.
  In the following, condition the process on the value of $s$.
  By the memorylessness of the exponential distribution,
  the conditional probability that the new signal is excitatory is
  \[ \frac{N^+_s(v)-S^+(v,t)}{N_s(v)-S(v,t)}\text. \]
  By our choice of  $C$, we deduce that (i) implies that
  \[ S^+(v,t)\leq S(v,t)< C_0 \leq \xi N_{C/p}(v)< \xi N_{s+1}(v)\leq \xi (N_s(v)+1)\text.\]
  Also, again by (i), we have
  \[\frac{N^+_s(v)}{N_s(v)} = \frac{N^+_s(v)}{N^+_s(v)+N^-_s(v)}\in (1\pm 3\xi)\beta \]
  for all $s\geq C/p$.
  Therefore, if $\xi$ is small enough,
  the probability of an excitatory signal is at least
  \[ \frac{(1-3\xi)\beta N_s(v)-\xi(N_s(v)+1)}{N_s(v)}\geq \beta -\eta \]
  and at most
  \[ \frac{(1+3\xi)\beta N_s(v)}{N_s(v)-\xi (N_s(v)+1)}\leq \beta+\eta\text. \]
  Note that these bounds hold independently of the value of $s$.

  Now consider the case $\tau \geq 1/(1+\gamma)$, which is equivalent to $\beta \geq 1/2$. For every $\eta>0$,
  and every large enough $C_0$ (depending on $\eta$),
  by a similar argument as above, with probability $1-\xi$, there exist a set $V_0$ of
  $(1-\xi)n$ vertices such that for every vertex $v\in V_0$, each of the first $C_0$ signals
  arriving in $v$ is excitatory with probability at least $1/2-\eta$, and moreover, $S(v,t_{1/p}+\xi)\geq C_0$.
  Then, if $C_0$ is large enough, Lemma~\ref{lemma:random-walk-app} and the Chernoff bound show that a.a.s.\ at least
  $(1-\xi)\left(\frac{1/2 -\eta}{1/2+\eta}\right)^k|V_0|$ vertices of $V_0$
  become active with one of the first $C_0$ signals, and for sufficiently small $\xi$ and $\eta$ this number is at least $n-\zeta n$. Since for the vertices of $V_0$, we have
  $S(v,t_{1/p}+\xi)\geq C_0$, and since $t_{1/p}+\xi$ can be upper bounded by a  constant $T$ (by Lemma~\ref{lemma:asyncexplosion}),
  this shows that $a(T)\geq n-\zeta n$.
\end{proof}

It remains to prove that if $\tau\geq 1/(1+\gamma)$, then
not too many vertices become active. This is the content of the following lemma.

\begin{lemma}\label{lemma:upperbound}
  For every $\eps>0$,
  the asynchronous percolation process with $\tau\geq  1/(1+\gamma)$,
  $(1+\eps)\threshold \leq \start \ll n$,
  and $n^{-1}\ll p\ll n^{-1/k}$ a.a.s.~does not activate more than $n\beta^k/(1-\beta)^k +o(n)$
  vertices, where $\beta$ is given by~\eqref{eq:beta}.
\end{lemma}

\begin{proof}
  If $\tau= 1/(1+\gamma)$, the statement is trivial since then $\beta = 1/2$. So assume that $\tau > 1/(1+\gamma)$ and thus $\beta < 1/2$.

  Let $\zeta \in (0,1)$ be arbitrarily small, but fixed. Let $\xi=\xi(\zeta)>0$ be sufficiently small (to be fixed below),
  and let $C_0= C_0(\zeta)>0$ be sufficiently large (so that we can apply~Lemma~\ref{lemma:walks} and the inequalities below hold). Lastly, assume that $n$ is sufficiently large. Let $V_0$ be a set of  $n-\zeta n$ vertices such that
  \begin{equation}\label{eq:ugly} N^+_s(v) \leq (1+\xi)\beta N_s(v) \end{equation}
  holds for every $v\in V_0$ and $s$ such that $S(v,t_{s})>0$.
  If $n$ is large enough,
  then such a set exists with probability at least $1-\zeta$, by
  Lemma~\ref{lemma:endgame} (i) and (ii).

  Recall that by Lemma~\ref{lemma:walks}, for every $\zeta >0$, for sufficiently large $n$ with probability at least $1-\zeta$ at most $\beta^kn/(1-\beta)^k + \zeta n$ vertices get activated by their first $C_0$ incoming signals. Thus, we only need to to show that there are few vertices that get activated later than by the first $C_0$ signals. More precisely, we will show that for large enough $n$, with probability at least $1-3\zeta$ there are at most $3\zeta n$ vertices $v$ for which there is $i>C_0$ such that $Z_i^{(v)} = k$.

  Again by Lemma~\ref{lemma:walks} with probability $1-\zeta$ there is a set $U_0\subseteq V_0$ of size $|U_0| \geq (1-2\zeta)n$ such that $Z^{(v)}_{C_0} \leq (1-\zeta)(2\beta-1)C_0$ for all $v \in U_0$, for sufficiently large $n$. The proof will be complete
  if we show that for every $v\in U_0$, the probability that $Z_i^{(v)}=k$
  holds for some $i>C_0$ is at most $\zeta^2$ (the statement then follows from Markov's inequality
  and the requirement that only $\start\ll n$ vertices are active initially).
  Given some $v\in V_0$, write $\mathcal A_{i,i^*}^{(v)}$ for the event that
  \begin{enumerate}[(i)]
    \item $Z_i^{(v)}\leq (1-\zeta)(2\beta-1)i$,
    \item $Z_{i^*}^{(v)}= k$, and
    \item $Z_j^{(v)}> (1-\zeta)(2\beta-1)j$ for all $i< j < i^*$.
  \end{enumerate}
  Then it suffices to show that $\Pr[\bigcup_{i^*> i\geq C_0}\mathcal A_{i,i^*}^{(v)}]\leq \zeta^2$. We first show that for every $v\in U_0$ and $i\geq C_0$, we have
  \[ \Pr\Big[\bigcup_{i^*> i} \mathcal A_{i,i^*}^{{(v)}}\Big]\leq \left(\frac{\beta+\xi}{1-\beta-\xi}\right)^{(1-\zeta)(1-2\beta)i+k}.\]
  For this, assume that $j > C_0$ is such that $Z_{j}^{(v)}>(1-\zeta)(2\beta-1)j$. Let $t$ be any time
  at which $v$ has seen exactly $j-1$ signals and assume that the $j$-th signal arrives
  in $v$ at time $t^*\in (t_s,t_{s+1}]$.
  In the following, condition on the value of $s$.
  By \eqref{eq:ugly}, we know that $N^+_s(v)\leq (1+\xi)\beta N_s(v)$.
  Moreover, from $Z_{j}^{(v)} =2 S^+(v,t^*)-j$ and our assumption on $Z_{j}^{(v)}$ we get 
  \[S^+(v,t)\geq S^+(v,t^*)-1\geq \beta j -\zeta(2\beta-1)j/2-1
  \geq (1+\xi)\beta S(v,t),\]
  where for the last inequality, we used the facts that $S(v,t)=j-1$, that $2\beta -1$ is negative, that $j> C_0$, where $C_0$ is sufficiently large, and the fact that we may choose $\xi=\xi(\zeta)$ to be sufficiently small.
  Then we have
  \[ \Pr[ Z^{(v)}_{j+1}-Z^{(v)}_j >0] = \frac{N^+_{s}(v)-S^+(v,t)}{N_{s}(v)-S(v,t)}
  \leq \frac{(1+\xi)\beta (N_{s}(v)-S(v,t))}{N_{s}(v)-S(v,t)}
  \leq \beta +\xi\text.\]
  By Lemma~\ref{lemma:random-walk-app}, this means that for every $i\geq C_0$ such that
  (i) holds, the probability that there exists some $i^*$ for which (ii) and (iii) hold is at most
  \[ \left(\frac{\beta+\xi}{1-\beta-\xi}\right)^{(1-\zeta)(1-2\beta)i+k}\text,\]
  where, since $\beta<1/2$, we can assume that $(\beta+\xi)/(1-\beta-\xi)< 1$.
  For large enough $C_0$, the union bound yields
  \[\Pr\Big[\bigcup_{i^*>i\geq C_0}\mathcal A_{i,i^*}\Big] \leq \sum_{i\geq C_0} \left(\frac{\beta+\xi}{1-\beta-\xi}\right)^{(1-\zeta)(1-2\beta)i}\leq \zeta^2,\]
  completing the proof.
\end{proof}

\begin{proof}[Proof of Theorem~\ref{thm:async-percolation}]
  The statement for $\tau\geq1/(1+\gamma)$ is an immediate consequence of Lemmas~\ref{lemma:walks} and~\ref{lemma:upperbound}.

  For the case $\tau < 1/(1+\gamma)$,
  we know by Lemma~\ref{lemma:walks} that if $T=T(\eps)$ is sufficiently large, then
  $s=n-o(n)$ vertices are active at time $T$, which takes
  care of the first part of this case.

  For the case $\tau<1/(1+\gamma)$ and  $p=\omega(\log n/n)$, the Chernoff and union bounds show that
  a.a.s.~every vertex $v\in [n]$ is such that $N^+_s(v)=(1+o(1))(1-\tau)pn$ and $N^-_n(v)=(1+o(1))\tau\gamma p n$.
  Using again the Chernoff and union bounds,
  within a time period of  constant length depending on $\tau<1/(1+\gamma)$,
  every vertex receives at least $|N_n^-(v)|+k$ excitatory signals and becomes active.
\end{proof}

\bibliographystyle{abbrv}
\bibliography{percolation}

\end{document}